\documentclass[12pt]{amsart}
\usepackage{amsmath,amssymb,amsthm}
\usepackage[a4paper,margin=3cm]{geometry}
\usepackage{enumitem}
\usepackage{xcolor}
\usepackage{hyperref}
\hypersetup{
  colorlinks=true,
  linkcolor=cyan,
  citecolor=cyan,
  urlcolor=cyan,
  hypertexnames=false
}
\newtheorem{theorem}{Theorem}
\newtheorem{lemma}[theorem]{Lemma}
\newtheorem{proposition}[theorem]{Proposition}
\newtheorem{corollary}[theorem]{Corollary}
\theoremstyle{definition}
\newtheorem{definition}[theorem]{Definition}
\newtheorem{remark}[theorem]{Remark}
\numberwithin{theorem}{section}
\numberwithin{equation}{section}
\newcommand{\G}{\mathbf G}
\newcommand{\Q}{\mathbb Q}
\newcommand{\Z}{\mathbb Z}
\newcommand{\R}{\mathbb R}
\newcommand{\CC}{\mathbb C}
\newcommand{\PGL}{\mathrm{PGL}}
\newcommand{\PSL}{\mathrm{PSL}}
\newcommand{\GL}{\mathrm{GL}}
\newcommand{\Span}{\mathrm{Span}}
\newcommand{\SL}{\mathrm{SL}}
\newcommand{\Gal}{\operatorname{Gal}}
\newcommand{\Prob}{\operatorname{Prob}}
\newcommand{\Supp}{\operatorname{Supp}}
\newcommand{\Ad}{\operatorname{Ad}}

\title{The $C^*$--ISR Property for $\PSL_n(\Z)$}
\author[Amrutam]{Tattwamasi Amrutam}
\address{Institute of Mathematics of the Polish Academy of Sciences, ul. Sniadeckich 8, 00-656, Warszawa, Poland}
\email{tattwamasiamrutam@gmail.com}
\thanks{The author is supported by National Science Centre, Poland Sonata, grant number 2025/59/D/ST1/03117.}
\date{\today}
\begin{document}
\begin{abstract}
A discrete group $\Gamma$ has the \emph{$C^*$--ISR property} if every unital
$\Gamma$--invariant $C^*$--subalgebra $A\subseteq C_r^*(\Gamma)$ is of the form
$C_r^*(N)$ for a normal subgroup $N\triangleleft\Gamma$. We show that $\PSL_n(\Z)$ satisfies the $C^*$--ISR property for every $n\geq3$.
\end{abstract}
\maketitle
\section{Introduction}
Let $\Gamma$ be a discrete group and let $L(\Gamma)$ be its group von Neumann
algebra. Inspired by Popa's deformation/rigidity theory, substantial progress
has been made on the structural classification of group von Neumann algebras.
One line of inquiry has been to describe the \emph{invariant} von Neumann
subalgebras, namely those $M\subseteq L(\Gamma)$ with $\lambda(g)\,M\,\lambda(g)^*\subseteq M$ for every $g\in\Gamma$.
A group is said to have the \emph{ISR property} if every such $M$ is of the
form $L(N)$ for a normal subgroup $N\triangleleft\Gamma$. Motivated by the works of Alekseev-Brugger~\cite{AB21}, Chifan-Das~\cite{CD20}, and Kalantar--Panagopoulos~\cite{KalantarPanagopoulos}, this direction has
attracted considerable attention (see, for example
\cite{AJ23,JZ24, DJ24, ADJS25, JL26a, JL26b, jiang2026factors, Amr26, amrutam2026relative} and the references therein).
The corresponding question for the reduced $C^*$--algebra is harder, and the
reason is structural. 

In the von Neumann setting, one has
the canonical trace-preserving conditional expectation onto any (invariant)
subalgebra, the Hilbert space $L^2(\Gamma)$ with its Fourier decomposition, and the bimodule techniques (employed by Jiang in \cite{jiang2021maximal}) that come with them. In the $C^*$--setting, none of
these is available. A unital $\Gamma$--invariant $C^*$--subalgebra $A\subseteq
C_r^*(\Gamma)$ need not be the range of any conditional expectation, and the
Fourier series of an element of $C_r^*(\Gamma)$ need not converge in norm.
Every argument has to be rebuilt so that it produces norm estimates directly.

Recall that a discrete group $\Gamma$ has the \emph{$C^*$--ISR property} if every unital
$\Gamma$--invariant $C^*$--subalgebra $A\subseteq C_r^*(\Gamma)$ is of the form
$C_r^*(N)$ for a normal subgroup $N\triangleleft\Gamma$. In the setup of negatively curved groups, the author and Jiang showed in
\cite{AJ26} that the $\Gamma$--invariant unital $C^*$--subalgebras of $C_r^*(\Gamma)$ for negatively curved groups and their free products, come from normal subgroups. On the other hand, Kalantar and Panagopoulos
\cite{KalantarPanagopoulos} proved that for a lattice $\Gamma$ in a semisimple
Lie group with no compact factors, trivial center, and real rank at least two,
every $\Gamma$--invariant $C^*$--subalgebra $A\subseteq C_r^*(\Gamma)$ is of the
form $C_r^*(N)$ \emph{provided that there is an equivariant conditional expectation}
$E:C_r^*(\Gamma)\to A$.
Our main result removes the expectation hypothesis in the higher-rank case for
$\PSL_n(\Z)$.
\begin{theorem}\label{thm:main}
Let $n\geq3$ and let $\Gamma=\PSL_n(\Z)$. Then every unital
$\Gamma$--invariant $C^*$--subalgebra $A\subseteq C_r^*(\Gamma)$ is of the form
$C_r^*(N)$ for a normal subgroup $N\triangleleft\Gamma$. That is, $\Gamma$ has
the $C^*$--ISR property.
\end{theorem}
Since $C_r^*(N)$ is the range of the canonical conditional expectation
$E_N:C_r^*(\Gamma)\to C_r^*(N)$, the hypothesis of
\cite{KalantarPanagopoulos} is not an assumption in our setting, but a
conclusion. The lack
of a conditional expectation means that the von Neumann techniques cannot be
transplanted directly, so, as in \cite{AJ26}, we use averaging to control the behaviour of individual elements of $A$, and, as has been the case since
\cite{BekkaCowlingHarpe, haagerup2015new}, the averaging comes from dynamics --- here, from
the action of $\PSL_n(\Z)$ on the projective space $\mathbb P^{n-1}(\R)$. We discuss the proof strategy in detail in Section~\ref{sec:gen-strategy}.
\section{General strategy}\label{sec:gen-strategy}
In \cite{AJ26}, the first step used the boundary action
$\Gamma\curvearrowright\partial\Gamma$, which for the groups considered there
is $2$--filling in the sense of \cite{Jolissaint}.
Two features of that situation made the argument work. The averaging
$\Ad(\lambda(t^{-k}))$ along a single element $t$ converged to the conditional
expectation $E_{C_r^*(C_\Gamma(t))}$, giving
$E_{C_r^*(C_\Gamma(t))}(A)\subseteq A$ for every primitive $s$. This was followed by a comparison inside a free subgroup (see \cite[Lemma~2.3]{AJ26}), which then isolated a group unitary. Both features
fail for $\PSL_n(\Z)$, and it is worth being precise about how.

First, the action of $\Gamma=\PSL_n(\Z)$ on $\mathbb P^{n-1}(\R)$ is
$n$--filling and not $2$--filling for $n\geq3$. A single pair of points does
not have $\Gamma$--translates covering the space, so the averaging cannot be
driven by two points at a time. Second, and more seriously, not every element of $\Gamma$ has north--south
dynamics on $\mathbb P^{n-1}(\R)$; in the language of Subsection~\S\ref{ssec:proximal},
not every element is proximal.
Third, even when averaging is available, the free-group comparison does not
get off the ground. For an element $p$ that our averaging \emph{can} control,
the centralizer is a maximal torus of rank $n-1$, so that $C_r^*\bigl(C_\Gamma(p)\bigr)\cong C(\mathbb T^{\,n-1})$, a commutative algebra of dimension $n-1$. There is nothing to compare against. As such, the heart of the matter boils down to the following abstract proposition.

We write $E_H$ for the canonical conditional expectation of $C_r^*(\Gamma)$ onto
$C_r^*(H)$. It is characterized on the canonical
unitaries by $E_H(\lambda(t))=\lambda(t)$ for $t\in H$ and
$E_H(\lambda(t))=0$ for $t\notin H$.
\begin{proposition}\label{prop:strategy-main}
Let $p,q,s\in\Gamma$ be such that the following hold true.
\begin{enumerate}
    \item[(A)] $q=ps^m$ for some $m\geq1$.
    \item[(B)] $C_\Gamma(p)\cap C_\Gamma(q)
 =C_\Gamma(q)\cap C_\Gamma(s)
 =C_\Gamma(p)\cap C_\Gamma(s)=\{e\}$.
 \item[(C)] For every $g\in C_\Gamma(p)$ and every
$t\in\Gamma\setminus C_\Gamma(s)$, one has $g\,s^{-k}t\,s^k\in C_\Gamma(q)$
for only finitely many integers $k\geq0$.
\end{enumerate}
Then for every $b\in C_r^*(C_\Gamma(p))$ and every $a\in C_r^*(\Gamma)$, $E_{C_\Gamma(q)}\bigl(b\,\lambda(s^{-k})\,a\,\lambda(s^{k})\bigr)
 \xrightarrow[\ k\to\infty\ ]{\ \|\cdot\|\ }
 E_{C_\Gamma(q)}\bigl(b\,E_{C_\Gamma(s)}(a)\bigr)$.
\end{proposition}
The proof rests on a Fourier expansion, and it deserves emphasis. In general, the Fourier series of an element
$a\in C_r^*(\Gamma)$ does not converge in norm. We need the approximation property for the Fourier series to converge, and the higher rank lattices do not have that by a fundamental result of Lafforgue and de la Salle~\cite{lafforgue2011noncommutative}. However, if
$b\in C_r^*(C_\Gamma(p))$ and $x\in C_r^*(C_\Gamma(s))$ with
$C_\Gamma(p)\cap C_\Gamma(s)=\{e\}$, and if $q=ps^m$, then the surviving part of the product does converge. We isolate this as the following lemma.
\begin{lemma}\label{lem:fourier-matching}
Let $p,q,s\in\Gamma$ be such that the assumptions of Proposition~\ref{prop:strategy-main} hold, and let
$ \Sigma=\bigl\{(c,h)\in C_\Gamma(p)\times C_\Gamma(s):ch\in C_\Gamma(q)\bigr\}$.
Then the following is true.
\begin{enumerate}
\item[(a)] If $b\in C_r^*(C_\Gamma(p))$ and
$x\in C_r^*(C_\Gamma(s))$, then we can write
$E_{C_\Gamma(q)}(bx)
 =\sum_{(c,h)\in\Sigma}
 \tau_0(b\lambda(c)^*)\tau_0(x\lambda(h)^*)\lambda(ch)$,
and this series converges in norm.
\item[(b)]  Moreover,
\[
\begin{aligned}
&\left\|E_{C_\Gamma(q)}(bx)-
 \tau_0(b\lambda(p)^*)\tau_0(x\lambda(s^m)^*)\lambda(q)\right\|\\
&\quad\leq
 \left\|b-\tau_0(b\lambda(p)^*)\lambda(p)\right\|_2
 \left\|x-\tau_0(x\lambda(s^m)^*)\lambda(s^m)\right\|_2.
\end{aligned}
\]
\end{enumerate}
\end{lemma}
\begin{proof}
By our assumption, the three groups $C_\Gamma(p)$,
$C_\Gamma(q)$ and $C_\Gamma(s)$ have pairwise trivial intersections.
Now, suppose that
$(c,h_1),(c,h_2)\in\Sigma$. Then $(ch_1)^{-1}(ch_2)=h_1^{-1}h_2$ belongs to
$C_\Gamma(q)\cap C_\Gamma(s)=\{e\}$. Hence $h_1=h_2$. If $(c_1,h),(c_2,h)\in\Sigma$, then
$(c_1h)(c_2h)^{-1}=c_1c_2^{-1}$ belongs to
$C_\Gamma(q)\cap C_\Gamma(p)=\{e\}$. Hence $c_1=c_2$.
We shall also use the injectivity of multiplication on
$C_\Gamma(p)\times C_\Gamma(s)$. If $c_1h_1=c_2h_2$, then
$c_2^{-1}c_1=h_2h_1^{-1}$ belongs to
$C_\Gamma(p)\cap C_\Gamma(s)=\{e\}$. Thus $c_1=c_2$ and $h_1=h_2$.
For $g\in\Gamma$, the Fourier coefficient of the product is
\[
 \tau_0(bx\lambda(g)^*)
 =\sum_{c\in C_\Gamma(p)}
 \tau_0(b\lambda(c)^*)
 \tau_0(x\lambda(c^{-1}g)^*).
\]
The scalar series is absolutely convergent by Cauchy--Schwarz. A summand
can be nonzero only when $c^{-1}g\in C_\Gamma(s)$, or equivalently when
$g=ch$ with $c\in C_\Gamma(p)$ and $h\in C_\Gamma(s)$. The injectivity of
multiplication shows that there is at most one such pair. After applying
$E_{C_\Gamma(q)}$, only the pairs in $\Sigma$ remain.
The full Fourier series of $b$ and $x$ need not converge in operator norm.
However, injectivity of the two coordinate projections and Cauchy--Schwarz inequality
give
\[
\begin{aligned}
&\sum_{(c,h)\in\Sigma}
 \left|\tau_0(b\lambda(c)^*)\tau_0(x\lambda(h)^*)\right|\\
&\quad\leq
 \left(\sum_{c\in C_\Gamma(p)}|\tau_0(b\lambda(c)^*)|^2\right)^{1/2}
 \left(\sum_{h\in C_\Gamma(s)}|\tau_0(x\lambda(h)^*)|^2\right)^{1/2}
 =\|b\|_2\|x\|_2.
\end{aligned}
\]
Since every $\lambda(ch)$ has norm one, this proves operator-norm
convergence of the matched subseries. The claim that it converges to
$E_{C_\Gamma(q)}(bx)$ follows from the $\|\cdot\|_2$--norm computation
above.
The pair $(p,s^m)$ belongs to $\Sigma$, because $ps^m=q$. Remove this
pair. If a remaining pair had $c=p$, injectivity on the first coordinate would force
$h=s^m$. Similarly, if it had $h=s^m$, injectivity of the second coordinate would force
$c=p$. Thus, every remaining pair satisfies both $c\neq p$ and
$h\neq s^m$. Using the triangle inequality, the two injective coordinate
projections once more, and Cauchy--Schwarz, we see that
\[
\begin{aligned}
&\left\|E_{C_\Gamma(q)}(bx)-
 \tau_0(b\lambda(p)^*)\tau_0(x\lambda(s^m)^*)\lambda(q)\right\|\\
&\quad\leq
 \left(\sum_{c\neq p}|\tau_0(b\lambda(c)^*)|^2\right)^{1/2}
 \left(\sum_{h\neq s^m}|\tau_0(x\lambda(h)^*)|^2\right)^{1/2}\\
&\quad=\|b-\tau_0(b\lambda(p)^*)\lambda(p)\|_2\|x-\tau_0(x\lambda(s^m)^*)\lambda(s^m)\|_2.
\end{aligned}
\]
This proves the estimate.
\end{proof}
\begin{proof}[Proof of Proposition~\ref{prop:strategy-main}]
Let $b\in C_r^*(C_\Gamma(p))$, let $a\in C_r^*(\Gamma)$, and fix
$\varepsilon>0$. Put
$M=\max\{\|a\|+1,\|b\|+1\}$. Choose finite Fourier sums $b_0$ and
$a_0$ such that $\|b-b_0\|<\varepsilon/(8M)$ and
$\|a-a_0\|<\varepsilon/(8M)$. Replace the coefficient of the identity in $b_0$ by $\tau_0(b)$; since
$\tau_0$ has norm one, that coefficient moves by at most $\|b-b_0\|$, so the
new finite Fourier sum is still within $\varepsilon/(4M)$ of $b$. Do the same
for $a_0$. We therefore obtain finite sets
$F_b\subseteq C_\Gamma(p)\setminus\{e\}$ and
$F_a\subseteq\Gamma\setminus\{e\}$, together with scalars
$\beta_g$ and $\alpha_h$, such that
\begin{equation}\label{eq:approx-b}
 \left\|b-\left(\sum_{g\in F_b}\beta_g\lambda(g)+\tau_0(b)1\right)\right\|
 <\frac{\varepsilon}{4M}
\end{equation}
and
\begin{equation}\label{eq:approx-a}
 \left\|a-\left(\sum_{h\in F_a}\alpha_h\lambda(h)+\tau_0(a)1\right)\right\|
 <\frac{\varepsilon}{4M}.
\end{equation}
For every $g\in F_b\cup\{e\}$ and every
$h\in F_a\setminus C_\Gamma(s)$, hypothesis (C) applies, because
$F_b\cup\{e\}\subseteq C_\Gamma(p)$, and yields an integer $k_0(g,h)$ such that
$gs^{-k}hs^k\notin C_\Gamma(q)$ whenever $k\geq k_0(g,h)$. If there are no such pairs, put $k_0=0$. Otherwise, let $k_0$ be
the maximum of these finitely many integers. Then, for every $k\geq k_0$,
\begin{equation}\label{eq:escape}
 gs^{-k}hs^k\notin C_\Gamma(q)
 \quad
 \text{for all }g\in F_b\cup\{e\},\
 h\in F_a\setminus C_\Gamma(s).
\end{equation}
The terms of the finite approximation to $a$ that belong to
$C_\Gamma(s)$ commute with $s$. The terms outside $C_\Gamma(s)$ are
annihilated by $E_{C_\Gamma(q)}$ for $k\geq k_0$, by equation~\eqref{eq:escape}. Hence, for
such $k$, adding and subtracting the finite approximations gives
\begin{equation}\label{eq:decomposition}
\begin{aligned}
&E_{C_\Gamma(q)}\bigl(b\lambda(s^{-k})a\lambda(s^k)\bigr)
 -E_{C_\Gamma(q)}\bigl(bE_{C_\Gamma(s)}(a)\bigr)\\
&=E_{C_\Gamma(q)}\left(
 \left[b-\left(\sum_{g\in F_b}\beta_g\lambda(g)+\tau_0(b)1\right)\right]
 \lambda(s^{-k})a\lambda(s^k)\right)\\
&\quad+E_{C_\Gamma(q)}\left(
 \left(\sum_{g\in F_b}\beta_g\lambda(g)+\tau_0(b)1\right)
 \lambda(s^{-k})
 \left[a-\left(\sum_{h\in F_a}\alpha_h\lambda(h)+\tau_0(a)1\right)\right]
 \lambda(s^k)\right)\\
&\quad+E_{C_\Gamma(q)}\left(
 \left[\left(\sum_{g\in F_b}\beta_g\lambda(g)+\tau_0(b)1\right)-b\right]
 \left(\sum_{h\in F_a\cap C_\Gamma(s)}\alpha_h\lambda(h)+\tau_0(a)1\right)
 \right)\\
&\quad+E_{C_\Gamma(q)}\left(
 b\left[\left(\sum_{h\in F_a\cap C_\Gamma(s)}\alpha_h\lambda(h)+\tau_0(a)1\right)
 -E_{C_\Gamma(s)}(a)\right]\right).
\end{aligned}
\end{equation}
Applying $E_{C_\Gamma(s)}$ to equation~\eqref{eq:approx-a} and using contractivity gives
\[
 \left\|
 \left(\sum_{h\in F_a\cap C_\Gamma(s)}\alpha_h\lambda(h)+\tau_0(a)1\right)
 -E_{C_\Gamma(s)}(a)
 \right\|<\frac{\varepsilon}{4M}.
\]
Consequently,
\[
 \left\|\sum_{h\in F_a\cap C_\Gamma(s)}\alpha_h\lambda(h)+\tau_0(a)1\right\|
 <\|a\|+1.
\]
Similarly, equation~\eqref{eq:approx-b} gives
\[
 \left\|\sum_{g\in F_b}\beta_g\lambda(g)+\tau_0(b)1\right\|
 <\|b\|+1.
\]
The four terms on the right-hand side of equation~\eqref{eq:decomposition} are therefore bounded,
respectively, by
 $\frac{\varepsilon}{4M}\|a\|$,
 $(\|b\|+1)\frac{\varepsilon}{4M}$,
 $\frac{\varepsilon}{4M}(\|a\|+1)$, and
 $\|b\|\frac{\varepsilon}{4M}$.
Since $M=\max\{\|a\|+1,\|b\|+1\}$, each of these four quantities is at
most $\varepsilon/4$. Thus, for every $k\geq k_0$,
\[
 \left\|E_{C_\Gamma(q)}\bigl(b\lambda(s^{-k})a\lambda(s^k)\bigr)
 -E_{C_\Gamma(q)}\bigl(bE_{C_\Gamma(s)}(a)\bigr)\right\|<\varepsilon.
\]
This proves the norm convergence.
\end{proof}
It remains to produce $p$, $q$ and $s$. The unipotent $s$ is fixed once and
for all as the image of $I_n+E_{12}$.
Section~\ref{sec:notandprelim} develops the geometry and produces the elements.
One further difference from \cite{AJ26} is worth recording. There, the proof of
the $C^*$--ISR property produced a new proof of the ISR property as a
by-product. Here the logic runs the other way. Our argument uses at the outset
that $\PSL_n(\Z)$ already satisfies the ISR property, through
\cite{KalantarPanagopoulos}, to know that $A''=L(N)$ for a normal subgroup $N$. We use this to isolate an element
$g\in\Gamma\setminus\{e\}$ with $\lambda(g)\in A$.
Let $N$ be the normal closure
of $g$. By Margulis's normal subgroup theorem, $N$ has finite index, so
$C_r^*(N)\subseteq A\subseteq C_r^*(\Gamma)$ with $\Gamma/N$ finite, and a
Galois correspondence for reduced crossed products by finite groups
\cite{CameronSmithGalois,Bedos1991} identifies the intermediate algebras with
the intermediate subgroups. Hence $A=C_r^*(K)$ for some $N\leq K\leq\Gamma$.
\section{Notation and Preliminaries}
\label{sec:notandprelim}
\subsection{Notation}\label{ssec:notation}
Fix $n\geq3$ and put $\Gamma=\PSL_n(\Z)$. We write
$\operatorname{pr}:\SL_n(\Z)\to\Gamma$ for the quotient map. We also write
$\G=\PGL_{n/\Q}$. Thus, for every field extension $F/\Q$,
$\G(F)=\PGL_n(F)$, and we identify $\Gamma$ with a subgroup of $\G(\Q)$.
Lower-case letters denote elements of $\Gamma$, and the corresponding
upper-case letters denote chosen lifts in $\SL_n(\Z)$. Thus, if
$g\in\Gamma$, then $G\in\SL_n(\Z)$ denotes a lift of $g$.
We write $\lambda:\Gamma\to\mathcal U(\ell^2\Gamma)$ for the left regular
representation and $\tau_0$ for the canonical trace on $C_r^*(\Gamma)$. For
$a\in C_r^*(\Gamma)$, put
$\Supp(a)=\{t\in\Gamma:\tau_0(a\lambda(t)^*)\neq0\}$, and, for a
$C^*$--subalgebra $A$, put $\Supp(A)=\bigcup_{a\in A}\Supp(a)$. For $g\in\Gamma$, write
$C_\Gamma(g)=\{t\in\Gamma:tg=gt\}$. Its algebraic centralizer is denoted by
$C_\G(g)$; thus $C_\G(g)(F)=\{x\in\G(F):xg=gx\}$ for a field extension
$F/\Q$.
\subsection{Proximal endomorphisms}\label{ssec:proximal}
Let $V$ be a finite-dimensional real vector space. Following
\cite[\S4.1, p.~51--52]{BenoistQuint}, a nonzero $f\in\operatorname{End}(V)$
is \emph{proximal} if $f$ has a unique eigenvalue of maximal absolute value
and this eigenvalue is a simple root of the characteristic polynomial. Both
the eigenvalue and its eigenline are then defined over $\R$. We write
$V_f^+\subseteq V$ for that eigenline, the \emph{attracting line}, and
$V_f^<\subseteq V$ for the unique $f$--stable hyperplane with
$V_f^+\not\subseteq V_f^<$; concretely, $V_f^<$ is the sum of the generalized
eigenspaces of the remaining eigenvalues.

Equivalently
\cite[\S4.1]{BenoistQuint}, $f$ is proximal exactly when the induced map of
$\mathbb P(V)$ has an attracting fixed point, which is then $V_f^+$.
Scaling $f$ scales all its eigenvalues by the same factor and changes neither
$V_f^+$ nor $V_f^<$. So it makes sense to call an element of $\PGL(V)$
proximal, and to write $V_a^+$ and $V_a^<$ for it, when one, equivalently
every, lift is proximal.
If $v\in V$ and $\varphi\in V^*$, write $v\otimes\varphi$ for the rank-one
operator $w\mapsto\varphi(w)v$. Its kernel is $\ker\varphi$ and its image is
$\R v$. Its only possible nonzero eigenvalue is $\varphi(v)$. Indeed, if
$(v\otimes\varphi)(w)=\mu w$ with $\mu\neq0$ then $\varphi(w)v=\mu w$, so
$w$ is a nonzero multiple of $v$, and applying $\varphi$ to
$\varphi(w)v=\mu w$ gives $\varphi(w)\varphi(v)=\mu\varphi(w)$ with
$\varphi(w)\neq0$, whence $\mu=\varphi(v)$.
The next two lemmas prepare the only statement of this subsection that is used
later, namely Lemma~\ref{lem:proximal-product}. The first says that a proximal
matrix, suitably renormalized, converges to a rank-one operator; the second is
the form of the continuity of roots that we shall need.
\begin{lemma}\label{lem:rank-one-limit}
Let $A\in\GL(V)$ be proximal, let $\lambda_+$ be its eigenvalue of maximal
absolute value and let $v_+$ span $V_A^+$. Let $\varphi_+\in V^*$ be the
linear functional with $\varphi_+(v_+)=1$ and $\ker\varphi_+=V_A^<$. Then
 $\lambda_+^{-k}A^k\longrightarrow v_+\otimes\varphi_+$
in operator norm as $k\to\infty$.
\end{lemma}
\begin{proof}
The decomposition $V=\R v_+\oplus V_A^<$ is $A$--invariant, and every
eigenvalue of $A|_{V_A^<}$ has absolute value strictly smaller than
$|\lambda_+|$. Hence the spectral radius $\rho$ of $\lambda_+^{-1}A|_{V_A^<}$
satisfies $\rho<1$, and Gelfand's formula
$\|(\lambda_+^{-1}A|_{V_A^<})^k\|^{1/k}\to\rho$ shows that
$\|(\lambda_+^{-1}A|_{V_A^<})^k\|\to0$. On the other hand $\lambda_+^{-1}A$
is the identity on $\R v_+$. Since $v_+\otimes\varphi_+$ is also the identity
on $\R v_+$ and is zero on $V_A^<$, the difference
$\lambda_+^{-k}A^k-v_+\otimes\varphi_+$ vanishes on $\R v_+$ and equals
$(\lambda_+^{-1}A|_{V_A^<})^k$ on $V_A^<$, so its norm tends to $0$.
\end{proof}
The next lemma is the standard fact that the roots of a monic polynomial move
continuously with its coefficients. We include the proof because the
application below needs the precise form.
\begin{lemma}\label{lem:rouche}
Let $P_k$ and $P$ be monic polynomials of degree $n$ over $\CC$, and suppose
that the coefficients of $P_k$ converge to those of $P$. Let $\mu$ be a root
of $P$ of multiplicity $m$, and let $\varepsilon>0$ be such that $\mu$ is the
only root of $P$ in the closed disc of radius $\varepsilon$ about $\mu$. Then
there is $k_0$ such that for every $k\geq k_0$ the polynomial $P_k$ has
exactly $m$ roots, counted with multiplicity, in the open disc of radius
$\varepsilon$ about $\mu$.
\end{lemma}
\begin{proof}
The circle $\{|t-\mu|=\varepsilon\}$ is compact and $P$ has no zero on it, so
$\delta:=\min_{|t-\mu|=\varepsilon}|P(t)|>0$. The coefficientwise convergence
$P_k\to P$ is uniform convergence on that circle, so there is $k_0$ with
$|P_k(t)-P(t)|<\delta\leq|P(t)|$ for all $k\geq k_0$ and all $t$ on the
circle. By Rouch\'e's theorem $P_k=P+(P_k-P)$ and $P$ have the same number of
zeros, with multiplicity, inside the circle. That number is $m$.
\end{proof}
We can now say when multiplying a large power of a proximal element by a fixed
element preserves proximality.
\begin{lemma}\label{lem:proximal-product}
Let $a\in\PGL(V)$ be proximal and let $c\in\PGL(V)$ satisfy
$c\,V_a^+\not\subseteq V_a^<$. Then $a^kc$ is proximal for all sufficiently
large $k$.
\end{lemma}
\begin{proof}
Choose lifts $A$ and $C$, and keep the notation of
Lemma~\ref{lem:rank-one-limit}. Multiplication on the right by $C$ is
continuous for the operator norm, so that
 $M_k:=\lambda_+^{-k}A^kC$ converges to $v_+\otimes(\varphi_+\circ C)=:\pi$.
The hypothesis $cV_a^+\not\subseteq V_a^<$ says $Cv_+\notin\ker\varphi_+$,
that is, $\mu:=\varphi_+(Cv_+)\neq0$. By the computation preceding
Lemma~\ref{lem:rank-one-limit}, $\pi$ is a rank-one operator whose only
nonzero eigenvalue is $\mu$; hence its characteristic polynomial is
$P(t)=t^{n-1}(t-\mu)$, where $n=\dim V$.
The entries of $M_k$ converge to those of $\pi$, so the coefficients of the
characteristic polynomial $P_k$ of $M_k$, being polynomials in the entries,
converge to those of $P$. Fix $\varepsilon$ with
$0<\varepsilon<|\mu|/2$. Applying Lemma~\ref{lem:rouche} at the root $\mu$,
of multiplicity one, and at the root $0$, of multiplicity $n-1$, we find $k_0$
such that for $k\geq k_0$ the polynomial $P_k$ has exactly one root in the
disc of radius $\varepsilon$ about $\mu$ and exactly $n-1$ roots in the disc
of radius $\varepsilon$ about $0$. These two discs are disjoint and account
for all $n$ roots. The single root near $\mu$ has absolute value at least
$|\mu|-\varepsilon>|\mu|/2$, while each of the others has absolute value at
most $\varepsilon<|\mu|/2$. So $M_k$ has a unique eigenvalue of maximal
absolute value, and it is simple: $M_k$ is proximal. Finally $M_k$ is a
nonzero scalar multiple of $A^kC$, so $A^kC$ is proximal as well, and hence so
is $a^kc$.
\end{proof}
\begin{remark}\label{rem:bq-lemma41}
This is the argument used in the proof of
\cite[Lemma~4.1, p.~52]{BenoistQuint}, where it is phrased in terms of
attracting fixed points on $\mathbb P(V)$, and invoked again in the proofs of
\cite[Lemma~6.25, p.~99]{BenoistQuint} and
\cite[Lemma~6.37, p.~105]{BenoistQuint} under the words ``reasoning as in the
proof of Lemma~4.1''.
\end{remark}
\subsection{Exterior powers}\label{ssec:wedge}
Everything we do with an $\R$--regular element is done one exterior power at
a time, because $\R$--regularity of $g$ is exactly proximality of
$\bigwedge^jg$ for every $j$. 

For $1\leq j\leq n$ let $\bigwedge^j\R^n$ be the vector space with basis the
symbols
\[
 e_I=e_{i_1}\wedge\cdots\wedge e_{i_j},
 \qquad I=\{i_1<\cdots<i_j\}\subseteq\{1,\ldots,n\},
\]
carrying the unique multilinear map
$(v_1,\ldots,v_j)\mapsto v_1\wedge\cdots\wedge v_j$ which vanishes whenever
two of its arguments are equal and which sends $(e_{i_1},\ldots,e_{i_j})$ to
$e_I$ for $i_1<\cdots<i_j$. Expanding
$0=(u+w)\wedge(u+w)\wedge\cdots$ in the first two arguments shows that
interchanging two arguments changes the sign. For $B\in\GL_n(\R)$ let
$\bigwedge^jB$ be the linear map with
$(\bigwedge^jB)(v_1\wedge\cdots\wedge v_j)=Bv_1\wedge\cdots\wedge Bv_j$;
comparing values on the decomposable elements gives
$\bigwedge^j(AB)=(\bigwedge^jA)(\bigwedge^jB)$ and
$\bigwedge^j(I_n)=\mathrm{id}$, so $\bigwedge^jB$ is invertible.
\begin{lemma}\label{lem:wedge-indep}
Let $v_1,\ldots,v_j$ and $w_{j+1},\ldots,w_n$ be vectors of $\R^n$ and denote
$U=\Span\{v_1,\ldots,v_j\}$ and $U'=\Span\{w_{j+1},\ldots,w_n\}$. Vectors $v_1,\ldots,v_j\in\R^n$ are linearly independent if and only if
$v_1\wedge\cdots\wedge v_j\neq0$. Moreover, $$(v_1\wedge\cdots\wedge v_j)\wedge(w_{j+1}\wedge\cdots\wedge w_n)\neq0
 ~\text{in }\textstyle\bigwedge^n\R^n
 ~\Longleftrightarrow~\
 U\oplus U'=\R^n .$$
\end{lemma}
\begin{proof}
If they are dependent, then, after renumbering, $v_j=\sum_{i<j}c_iv_i$, and then
expanding by multilinearity in the last argument gives a sum of terms each
having a repeated argument, so the wedge vanishes. If they are independent, extend to a basis $v_1,\ldots,v_n$ and let
$B\in\GL_n(\R)$ be the map with $Be_i=v_i$. Then $\bigwedge^jB$ is invertible
and sends the basis $(e_I)$ to the family $(v_I)$, so $(v_I)$ is a basis of
$\bigwedge^j\R^n$; and $v_1\wedge\cdots\wedge v_j=v_{\{1,\ldots,j\}}$ is one
of its members.
For the second claim, observe that the left-hand
side says that the $n$ vectors $v_1,\ldots,v_j,w_{j+1},\ldots,w_n$ are
independent. As there are exactly $n$ of them, that is the same as saying
that they form a basis, which is the same as $U\oplus U'=\R^n$.
\end{proof}
\subsection{Zariski density}\label{ssec:zariski}
We use the Zariski topology only through the following definition and the
single lemma that follows it, so we state both in the most elementary form.
Identify $M_n(\R)$ with $\R^{n^2}$ by listing the entries of a matrix; a
\emph{polynomial function} on $M_n(\R)$ is then a polynomial in those $n^2$
entries.
\begin{definition}\label{def:zariski-dense}
Let $S\subseteq M_n(\R)$ and let $D\subseteq S$. We say that $D$ is
\emph{Zariski dense} in $S$ if every polynomial function which vanishes at
every point of $D$ vanishes at every point of $S$.
\end{definition}
\begin{lemma}\label{lem:dense-avoids}
Let $D$ be Zariski dense in $S$ and let $f_1,\ldots,f_m$ be polynomial
functions. Suppose that there is a point $y\in S$ with $f_i(y)\neq0$ for every
$i$. Then there is a point $d\in D$ with $f_i(d)\neq0$ for every $i$.
\end{lemma}
\begin{proof}
Put $f=f_1f_2\cdots f_m$, again a polynomial function. By hypothesis
$f(y)=f_1(y)\cdots f_m(y)\neq0$, so $f$ does not vanish at every point of $S$.
By Definition~\ref{def:zariski-dense}, $f$ therefore does not vanish at every
point of $D$: there is $d\in D$ with $f(d)\neq0$. A product of real numbers is
nonzero only if every factor is, so $f_i(d)\neq0$ for every $i$.
\end{proof}
The above is the only property of Zariski density that we
use. In topological language, it says that a dense set meets a nonempty open
set, and the passage from finitely many conditions to one is the reason no
irreducibility hypothesis is needed.
\subsection{Loxodromic elements}
\label{ssec:loxodromic}
\begin{definition}[{\cite[Definition~6.10, p.~93]{BenoistQuint}}]
\label{def:loxodromic}
An element $B\in\SL_n(\R)$ is \emph{loxodromic} if the absolute values of its
eigenvalues are pairwise distinct. An element $g\in\PGL_n(\R)$ is
\emph{$\R$--regular} if one, equivalently every, lift of $g$ to $\GL_n(\R)$
has pairwise distinct absolute values of eigenvalues; for $g\in\Gamma$ this
says that a lift in $\SL_n(\Z)$ is loxodromic.
\end{definition}
Scaling a lift multiplies all eigenvalues by the same factor, so the condition
does not depend on the lift. If $g$ is $\R$--regular then all eigenvalues of a
lift are real, since a nonreal eigenvalue occurs together with its complex
conjugate and the two have the same absolute value.
We quote the two results from \cite{BenoistQuint} that we use.
\begin{lemma}[{\cite[Lemma~6.27, p.~100]{BenoistQuint}}]
\label{lem:bq-lox-wedge}
An element $B\in\SL_n(\R)$ is loxodromic if and only if $\bigwedge^jB$ is
proximal in $\bigwedge^j\R^n$ for every $1\leq j\leq n-1$.
\end{lemma}
\begin{proposition}[{\cite[Proposition~6.11, p.~93]{BenoistQuint}}]
\label{prop:bq-lox-dense}
Let $\Delta$ be a Zariski-dense subsemigroup of $\SL_n(\R)$. Then the set of
loxodromic elements of $\Delta$ is Zariski dense in $\SL_n(\R)$. In
particular it is nonempty.
\end{proposition}
The next lemma computes, once and for all, what the attracting line and the
invariant hyperplane of $\bigwedge^jR$ are in terms of the eigenvectors of
$R$. It is the bridge between $\R$--regularity and the flags of
\S\ref{ssec:flags}.
\begin{lemma}\label{lem:wedge-attracting}
Let $r\in\Gamma$ be $\R$--regular, let $R$ be a lift with eigenvalues
$|\mu_1|>\cdots>|\mu_n|$ and corresponding eigenvectors $v_1,\ldots,v_n$, and
let $1\leq j\leq n-1$. Then $\bigwedge^jR$ is proximal, with
\[V^+_{\bigwedge^jR}=\R\,(v_1\wedge\cdots\wedge v_j),
 \text{ and }
 V^<_{\bigwedge^jR}
 =\bigl\{\xi:\ \xi\wedge(v_{j+1}\wedge\cdots\wedge v_n)=0\bigr\}.
\]
\end{lemma}
\begin{proof}
Since $Rv_i=\mu_iv_i$, the definition of $\bigwedge^jR$ gives
$(\bigwedge^jR)v_I=\bigl(\prod_{i\in I}\mu_i\bigr)v_I$, so $\bigwedge^jR$ is
diagonal in the basis $(v_I)$ of Lemma~\ref{lem:wedge-indep} with those
eigenvalues. Their absolute values are $\prod_{i\in I}|\mu_i|$, and since
$|\mu_1|>\cdots>|\mu_n|$ the product is strictly largest for the unique choice
$I=\{1,\ldots,j\}$. So $\bigwedge^jR$ has a unique eigenvalue of maximal
absolute value, it is simple, and the attracting line is
$\R\,v_{\{1,\ldots,j\}}$; the invariant complementary hyperplane is the span
of the remaining $v_I$. It remains to check that the span of the $v_I$ with
$I\neq\{1,\ldots,j\}$ is the set described in the statement. We verify the
two inclusions separately, after one computation. Write a general element of $\bigwedge^j\R^n$ as $\xi=\sum_Ic_Iv_I$. If
$I\neq\{1,\ldots,j\}$, then $I$ contains at least one index larger than $j$,
so the wedge $v_I\wedge(v_{j+1}\wedge\cdots\wedge v_n)$ repeats one of its
arguments and therefore vanishes. Only the term $I=\{1,\ldots,j\}$ survives,
so that
$\xi\wedge(v_{j+1}\wedge\cdots\wedge v_n)
 =c_{\{1,\ldots,j\}}\,(v_1\wedge\cdots\wedge v_n)$,
and $v_1\wedge\cdots\wedge v_n\neq0$ by Lemma~\ref{lem:wedge-indep}. Let $\xi\in\Span\{v_I:I\neq\{1,\ldots,j\}\}$. Then
$c_{\{1,\ldots,j\}}=0$, so the display gives
$\xi\wedge(v_{j+1}\wedge\cdots\wedge v_n)=0$. Now, if $\xi$ satisfies
$\xi\wedge(v_{j+1}\wedge\cdots\wedge v_n)=0$. The display then reads
$c_{\{1,\ldots,j\}}(v_1\wedge\cdots\wedge v_n)=0$, and since
$v_1\wedge\cdots\wedge v_n\neq0$ we get $c_{\{1,\ldots,j\}}=0$; that is,
$\xi$ is supported on the $v_I$ with $I\neq\{1,\ldots,j\}$.
\end{proof}
\subsection{Flags and transversality}\label{ssec:flags}
\begin{definition}\label{def:flag}
A \emph{complete flag} $F$ in $\R^n$ is a chain of subspaces
\[ 0=V_0\subset V_1\subset\cdots\subset V_{n-1}\subset V_n=\R^n,
~ \dim V_j=j .
\]
Two complete flags $F=(V_j)$ and $F'=(V'_j)$ are \emph{transverse}, written
$F\pitchfork F'$, if $V_j\oplus V'_{n-j}=\R^n$ for every $j$. The group
$\GL_n(\R)$ acts on flags by $gF=(gV_j)$, and the action factors through
$\PGL_n(\R)$ because scalars preserve every subspace.
\end{definition}
\begin{lemma}\label{lem:transverse-basic}
Transversality is symmetric, and for $g\in\GL_n(\R)$ one has
$F\pitchfork F'$ if and only if $gF\pitchfork gF'$.
\end{lemma}
\begin{proof}
Replacing $j$ by $n-j$ in the defining condition interchanges the roles of $F$
and $F'$, which gives symmetry. An invertible $g$ carries the decomposition
$\R^n=V_j\oplus V'_{n-j}$ to $\R^n=gV_j\oplus gV'_{n-j}$, and $g^{-1}$
carries it back.
\end{proof}
\begin{lemma}\label{lem:not-union}
A real vector space is not the union of finitely many proper subspaces.
\end{lemma}
\begin{proof}
Suppose $\R^n=U_1\cup\cdots\cup U_m$ with all $U_i$ proper, and choose such a
covering with $m$ smallest possible. By minimality, $U_2\cup\cdots\cup U_m\neq\R^n$ because otherwise deleting $U_1$ would
give a covering with $m-1$ pieces. Choose a point outside
$U_2\cup\cdots\cup U_m$; since the $U_i$ cover $\R^n$ it lies in $U_1$. Call
it $v$. Thus $v\in U_1$ and $v\notin U_i$ for every $i\geq2$; in particular
$v\neq0$. Since $U_1$ is proper, choose $w\in\R^n$ with $w\notin U_1$.
Consider the infinitely many vectors $w+tv$, $t\in\R$. None lies in $U_1$: if
$w+tv\in U_1$ then, as $v\in U_1$ and $U_1$ is a subspace,
$w=(w+tv)-tv\in U_1$, contrary to the choice of $w$. So each $w+tv$ lies in
some $U_i$ with $i\geq2$. There are finitely many such $i$ and infinitely many
$t$, so by the pigeonhole principle there are $t\neq t'$ and an index
$i\geq2$ with $w+tv\in U_i$ and $w+t'v\in U_i$. Subtracting,
$(t-t')v\in U_i$, and $t-t'\neq0$, so $v\in U_i$ with $i\geq2$. This
contradicts the choice of $v$.
\end{proof}
The following says that we can produce a flag in general position with
respect to finitely many given ones.
\begin{lemma}\label{lem:common-transverse}
Let $F^{(1)},\ldots,F^{(m)}$ be finitely many complete flags in $\R^n$. There
is a complete flag transverse to all of them.
\end{lemma}
\begin{proof}
Write $F^{(i)}=(V^{(i)}_j)$. We construct vectors $w_1,\ldots,w_n$ such that,
with $W_j=\Span\{w_1,\ldots,w_j\}$,
\begin{equation}\label{eq:greedy}
 W_j\cap V^{(i)}_{n-j}=0
 \qquad\text{for all } i\leq m \text{ and all } j\leq n,
\end{equation}
and such that $\dim W_j=j$. Granting this, $\dim W_j+\dim V^{(i)}_{n-j}=n$
together with trivial intersection gives $W_j\oplus V^{(i)}_{n-j}=\R^n$, so
$(W_j)$ is a complete flag transverse to every $F^{(i)}$. The construction is by induction on $j$. For $j=0$, there is nothing to do. Suppose $w_1,\ldots,w_{j-1}$ have been found, with $\dim W_{j-1}=j-1$ and with
\eqref{eq:greedy} valid up to index $j-1$. Fix $i$. The sum $W_{j-1}+V^{(i)}_{n-j}$ is direct. Indeed, since
$V^{(i)}_{n-j}\subseteq V^{(i)}_{n-j+1}$ we have
 $W_{j-1}\cap V^{(i)}_{n-j}\subseteq W_{j-1}\cap V^{(i)}_{n-j+1}=0$
by equation~\eqref{eq:greedy} at index $j-1$. Hence
$U_i:=W_{j-1}\oplus V^{(i)}_{n-j}$ has dimension $(j-1)+(n-j)=n-1$ and is a
proper subspace of $\R^n$. By Lemma~\ref{lem:not-union}, choose $w_j\notin U_1\cup\cdots\cup U_m$. In
particular $w_j\notin W_{j-1}$, so $\dim W_j=j$. Let us now verify equation~\eqref{eq:greedy} at index $j$. Let $x\in W_j\cap V^{(i)}_{n-j}$ and
write $x=u+tw_j$ with $u\in W_{j-1}$ and $t\in\R$. If $t\neq0$, solve for $w_j$ to get
$ w_j=t^{-1}x-t^{-1}u$. Here $x\in V^{(i)}_{n-j}\subseteq U_i$ and $u\in W_{j-1}\subseteq U_i$, and
$U_i$ is a subspace, so $w_j\in U_i$ --- contradicting the choice of $w_j$. Hence $t=0$, so $x=u\in W_{j-1}\cap V^{(i)}_{n-j}=0$. Thus
$W_j\cap V^{(i)}_{n-j}=0$, as required.
\end{proof}
The following says that a transverse pair of flags is unique up to the action
of the group, so that one such configuration can be moved to any other.
\begin{lemma}\label{lem:transitive-transverse}
$\SL_n(\R)$ acts transitively on the set of transverse pairs of complete flags
in $\R^n$.
\end{lemma}
\begin{proof}
It suffices to carry an arbitrary transverse pair to the fixed pair
$(F_{\mathrm{st}},F_{\mathrm{op}})$, where
$(F_{\mathrm{st}})_j=\Span\{e_1,\ldots,e_j\}$ and
$(F_{\mathrm{op}})_j=\Span\{e_n,\ldots,e_{n-j+1}\}$; two transverse pairs then
lie in a common orbit.
Let $(F,F')=\bigl((V_j),(V'_j)\bigr)$ be transverse and set
$L_i=V_i\cap V'_{n-i+1}$ for $1\leq i\leq n$.
We first claim that $V_i=V_{i-1}\oplus L_i$; in particular $\dim L_i=1$. We use the
following elementary fact. If $\R^n=X\oplus Y$ and $X\subseteq Z$ for a
subspace $Z$, then $Z=X\oplus(Z\cap Y)$. Indeed, given $z\in Z$ write
$z=x+y$ with $x\in X$ and $y\in Y$; then $y=z-x\in Z$ because $x\in X\subseteq Z$,
so $y\in Z\cap Y$ and $z\in X+(Z\cap Y)$; and
$X\cap(Z\cap Y)\subseteq X\cap Y=0$.
Transversality at index $i-1$ gives $\R^n=V_{i-1}\oplus V'_{n-i+1}$. Apply the
fact with $X=V_{i-1}$, $Y=V'_{n-i+1}$ and $Z=V_i$, which is legitimate as
$V_{i-1}\subseteq V_i$ to get
$ V_i=V_{i-1}\oplus\bigl(V_i\cap V'_{n-i+1}\bigr)=V_{i-1}\oplus L_i$.
Comparing dimensions, $i=(i-1)+\dim L_i$, so $\dim L_i=1$, proving the claim. Choose $0\neq u_i\in L_i$ for each $i$. From $V_0=0$ and
$V_i=V_{i-1}\oplus L_i$ we get, by induction on $j$,
$ V_j=\Span\{u_1,\ldots,u_j\}$ for $0\leq j\leq n$,
so in particular $u_1,\ldots,u_n$ is a basis of $\R^n$. Next, $u_i\in L_i\subseteq V'_{n-i+1}$. For $i\geq n-j+1$ we have
$n-i+1\leq j$, so $V'_{n-i+1}\subseteq V'_j$ and hence
$u_n,u_{n-1},\ldots,u_{n-j+1}\in V'_j$. These are $j$ independent vectors in
the $j$--dimensional space $V'_j$, so
$V'_j=\Span\{u_n,\ldots,u_{n-j+1}\}$.
Let $g\in\GL_n(\R)$ be the linear map determined by $gu_i=e_i$. By the two
displayed descriptions of $V_j$ and of $V'_j$ we have $gF=F_{\mathrm{st}}$
and $gF'=F_{\mathrm{op}}$. It remains to arrange that $\det g=1$, so that the element carrying the two
flags lies in $\SL_n(\R)$ and not merely in $\GL_n(\R)$.
The vectors $u_i$ were chosen only up to nonzero scalars. Fix
$\lambda\in\R^\times$ and replace $u_1$ by $\lambda u_1$, keeping
$u_2,\ldots,u_n$ unchanged. Let $g_\lambda$ be the linear map determined by
$g_\lambda(\lambda u_1)=e_1$ and $g_\lambda(u_i)=e_i$ for $i\geq2$. If
$d_\lambda$ is the linear map which sends $u_1$ to $\lambda^{-1}u_1$ and
fixes $u_2,\ldots,u_n$, then $g_\lambda=g\circ d_\lambda$. Consequently,
 $\det g_\lambda=\det g\,\det d_\lambda=\lambda^{-1}\det g$.
The line $L_1=\R u_1$ is unchanged by the rescaling, and so are all the spans
$\Span\{u_1,\ldots,u_j\}$ and
$\Span\{u_n,\ldots,u_{n-j+1}\}$. Therefore
$g_\lambda F=F_{\mathrm{st}}$ and $g_\lambda F'=F_{\mathrm{op}}$ still hold.
Choosing $\lambda=\det g$ gives $\det g_\lambda=1$, so $g_\lambda$ is the
required element of $\SL_n(\R)$.
\end{proof}
An $\R$--regular element $r$ carries two flags. The attracting flag
$F^+(r)$, whose $j$-th member is spanned by the $j$ eigenvectors belonging to
the $j$ largest eigenvalues in absolute value, and the repelling flag
$F^-(r)$, built in the same way from the other end. What has to be arranged,
separately in each exterior power, is that $\bigwedge^jC$ does not send the
attracting line of $\bigwedge^jR$ into its invariant hyperplane.
We show that these $n-1$ conditions, taken
together, say exactly that
 $cV_j\oplus W_{n-j}=\R^n$ for every $j$. This is to say that the flags $cF^+(r)$ and $F^-(r)$ are transverse in the sense of
Definition~\ref{def:flag}.
We use the phrase \emph{$c$ moves $F^+(r)$ into
general position with respect to $F^-(r)$} as a synonym for this
transversality.
\begin{lemma}\label{lem:flag-dictionary}
Let $r\in\Gamma$ be $\R$--regular with lift $R$, eigenvalues
$|\mu_1|>\cdots>|\mu_n|$ and eigenvectors $v_1,\ldots,v_n$. Let
$ F^+(r)$ denote the space $V_j=\Span\{v_1,\ldots,v_j\}$ and
 $F^-(r)$, the space  $W_j=\Span\{v_{n-j+1},\ldots,v_n\}$ which are transverse. Let $c\in\Gamma$ with lift $C$ and let
$1\leq j\leq n-1$. Then
\[\bigl(\textstyle\bigwedge^jC\bigr)
 \bigl(V^+_{\bigwedge^jR}\bigr)\not\subseteq V^<_{\bigwedge^jR}
 ~\Longleftrightarrow~
 cV_j\oplus W_{n-j}=\R^n .
\]
Consequently these conditions hold for every $j$ if and only if
$c\,F^+(r)\pitchfork F^-(r)$.
\end{lemma}
\begin{proof}
By Lemma~\ref{lem:wedge-attracting}, $V^+_{\bigwedge^jR}$ is spanned by
$v_1\wedge\cdots\wedge v_j$, so
$(\bigwedge^jC)(V^+_{\bigwedge^jR})$ is spanned by
$Cv_1\wedge\cdots\wedge Cv_j$; and the same lemma says that this vector lies
in $V^<_{\bigwedge^jR}$ exactly when $(Cv_1\wedge\cdots\wedge Cv_j)\wedge(v_{j+1}\wedge\cdots\wedge v_n)=0$. By Lemma~\ref{lem:wedge-indep}, the non-vanishing of that wedge is
equivalent to $cV_j\oplus W_{n-j}=\R^n$ (the condition does not depend on the
lift $C$, since replacing $C$ by $tC$ multiplies the wedge by $t^j$). Finally, $cF^+(r)$ has $j$-th member $cV_j$ and $F^-(r)$ has $(n-j)$-th member
$W_{n-j}$, so the conjunction over $j$ is Definition~\ref{def:flag} of
$cF^+(r)\pitchfork F^-(r)$. The cases $j=0$ and $j=n$ of that definition are
automatic.
\end{proof}
\begin{lemma}\label{lem:conjugate-flags}
Let $r\in\Gamma$ be $\R$--regular and let $\gamma\in\Gamma$. Then
$\gamma r\gamma^{-1}$ is $\R$--regular and
$F^\pm(\gamma r\gamma^{-1})=\gamma F^\pm(r)$.
\end{lemma}
\begin{proof}
Let $R$ and $Y$ be lifts of $r$ and $\gamma$. If $Rv=\mu v$ then
$(YRY^{-1})(Yv)=\mu\,Yv$, so $YRY^{-1}$ has the same eigenvalues as $R$, with
eigenvectors $Yv$. In particular the absolute values are still pairwise
distinct, so $\gamma r\gamma^{-1}$ is $\R$--regular, and the eigenbasis
ordered by decreasing absolute value of the eigenvalue is
$Yv_1,\ldots,Yv_n$. Both flags are the spans of initial, respectively final,
segments of that ordered basis, so both are the images under $\gamma$ of the
corresponding flags of $r$.
\end{proof}
The transversality condition, read as a condition on
the conjugating element, is the non-vanishing of finitely many polynomials.
\begin{lemma}\label{lem:transversality-open}
Let $v_1,\ldots,v_n$ be a basis of $\R^n$, let $C_1,\ldots,C_\ell\in\GL_n(\R)$
and, for $1\leq i\leq\ell$ and $1\leq j\leq n-1$, define
\[
 \Delta_{i,j}(Y)=\det\bigl[\,
 C_iYv_1\ \big|\ \cdots\ \big|\ C_iYv_j\ \big|\
 Yv_{j+1}\ \big|\ \cdots\ \big|\ Yv_n\,\bigr],
 \qquad Y\in M_n(\R).
\]
Then:
\begin{enumerate}
\item each $\Delta_{i,j}$ is a polynomial function of the entries of $Y$,
homogeneous of degree $n$;
\item for $Y\in\GL_n(\R)$ one has $\Delta_{i,j}(Y)\neq0$ if and only if
$c_iYV_j\oplus YW_{n-j}=\R^n$, where $V_j=\Span\{v_1,\ldots,v_j\}$ and
$W_{n-j}=\Span\{v_{j+1},\ldots,v_n\}$;
\item there exists $Y_0\in\SL_n(\R)$ with $\Delta_{i,j}(Y_0)\neq0$ for every
$i$ and every $j$.
\end{enumerate}
\end{lemma}
\begin{proof}
(1) Each column of the displayed matrix is a fixed matrix applied to $Y$
applied to a fixed vector, so each entry of the matrix is a linear form in the
entries of $Y$. The determinant of an $n\times n$ matrix whose entries are
linear forms is a homogeneous polynomial of degree $n$ in those variables.
(2) The columns are $c_iYv_1,\ldots,c_iYv_j$, which span $c_iYV_j$. Also, the columns $Yv_{j+1},\ldots,Yv_n$ span $YW_{n-j}$. The determinant is nonzero
exactly when the $n$ columns are independent, and by
Lemma~\ref{lem:wedge-indep} that is exactly the stated direct sum
decomposition.
(3) Let $F^+$ be the flag with $j$-th member $V_j$ and $F^-$ the flag with
$j$-th member $\Span\{v_{n-j+1},\ldots,v_n\}$; they are transverse. Apply
Lemma~\ref{lem:common-transverse} to the $\ell+1$ flags
$F^+,c_1F^+,\ldots,c_\ell F^+$ and obtain a flag $F''$ transverse to all of
them. Note that $F''\pitchfork F^+$ and $F''\pitchfork c_iF^+$, hence also
$F^+\pitchfork F''$ and $c_iF^+\pitchfork F''$ by the symmetry of
Lemma~\ref{lem:transverse-basic}. Both $(F^+,F^-)$ and $(F^+,F'')$ are
transverse pairs, so by
Lemma~\ref{lem:transitive-transverse} there is $y\in\SL_n(\R)$ with
$yF^+=F^+$ and $yF^-=F''$. Then for every $i$ and $j$,
\[
 c_i\,yV_j=c_iV_j
 \quad\text{and}\quad
 yW_{n-j}=(yF^-)_{n-j}=F''_{n-j},
\]
and $c_iF^+\pitchfork F''$ gives $c_iV_j\oplus F''_{n-j}=\R^n$. By part (2),
any lift $Y_0\in\SL_n(\R)$ of $y$ satisfies $\Delta_{i,j}(Y_0)\neq0$ for all
$i,j$.
\end{proof}
\subsection{\texorpdfstring{$\Q$}{Q}--generic elements}\label{ssec:qgeneric}
$\R$--regularity constrains the absolute values of the eigenvalues. We need a condition that constrains the Galois group of the characteristic
polynomial.
Let $f\in\Q[t]$ be monic of degree $n$ with distinct roots
$\alpha_1,\ldots,\alpha_n$ in a splitting field $L$. Every
$\sigma\in\Gal(L/\Q)$ permutes the roots, and we write $\pi_\sigma\in S_n$
for the permutation determined by $\sigma(\alpha_i)=\alpha_{\pi_\sigma(i)}$.
The homomorphism $\sigma\mapsto\pi_\sigma$ is injective, because $L$ is
generated over $\Q$ by the roots. For the Galois theory used here, see
Lang~\cite[Chapter~VI, \S2]{lang2012algebra}.
An element $g\in\Gamma$ is \emph{semisimple} if a lift $G\in\SL_n(\Z)$ is
diagonalizable over $\CC$, and \emph{regular semisimple} if moreover the $n$
eigenvalues of $G$ are pairwise distinct. Replacing a lift by its negative
multiplies every eigenvalue by $-1$, so neither condition depends on the lift.
\begin{definition}\label{def:qgeneric}
An element $g\in\Gamma$ is \emph{$\Q$--generic} if a lift $G\in\SL_n(\Z)$
has the following two properties:
\begin{enumerate}
\item the characteristic polynomial $\chi_G$ is irreducible over $\Q$;
\item if $L$ is the splitting field of $\chi_G$, then the homomorphism
$\Gal(L/\Q)\to S_n$, $\sigma\mapsto\pi_\sigma$, is surjective.
\end{enumerate}
\end{definition}
Two lifts of $g$ differ by a scalar matrix in the centre of $\SL_n(\Z)$, that
is, by a sign; replacing a lift by its negative replaces every root of the
characteristic polynomial by its negative and changes neither of the two
properties. So the definition does not depend on the lift. The next lemma explains what the two conditions say about the action of the
Galois group; both parts are used in Subsection~\S\ref{ssec:centralizers}.
\begin{lemma}\label{lem:galois}
Let $f$, $L$ and $\alpha_1,\ldots,\alpha_n$ be as above.
\begin{enumerate}
\item $f$ is irreducible over $\Q$ if and only if $\Gal(L/\Q)$ acts
transitively on $\{\alpha_1,\ldots,\alpha_n\}$.
\item Suppose the image of $\Gal(L/\Q)$ in $S_n$ is all of $S_n$, and let
$n\geq3$. Then for any $i\neq j$ and any $a\neq b$ there is
$\sigma\in\Gal(L/\Q)$ with $\pi_\sigma(i)=a$ and $\pi_\sigma(j)=b$; and if
$i\neq j$ then the stabilizers of $i$ and of $j$ in $S_n$ are distinct
subgroups.
\end{enumerate}
\end{lemma}
\begin{proof}
(1) If $f$ is irreducible then, for any $i$ and $j$, both $\Q(\alpha_i)$ and
$\Q(\alpha_j)$ are isomorphic to $\Q[t]/(f)$ over $\Q$, so there is an
isomorphism $\Q(\alpha_i)\to\Q(\alpha_j)$ sending $\alpha_i$ to $\alpha_j$,
and it extends to an automorphism of the splitting field $L$. Conversely, if
the action is transitive, let $h$ be the minimal polynomial of $\alpha_1$
over $\Q$. For each $i$ choose $\sigma$ with $\sigma\alpha_1=\alpha_i$;
applying $\sigma$ to $h(\alpha_1)=0$ and using that $h$ has rational
coefficients gives $h(\alpha_i)=0$. So $h$ has at least $n$ distinct roots,
whence $\deg h\geq n$ and $h=f$.
(2) For the first claim, extend $i\mapsto a$ and $j\mapsto b$ to a bijection
of $\{1,\ldots,n\}$, which is possible because the two pairs have the same
cardinality; the resulting permutation is $\pi_\sigma$ for some $\sigma$ by
surjectivity. For the second, choose $\ell\notin\{i,j\}$, which is possible
because $n\geq3$; the transposition $(j\ \ell)$ fixes $i$ and moves $j$, so
it lies in the stabilizer of $i$ but not in that of $j$.
\end{proof}
\begin{remark}\label{rem:qgeneric-pr}
We point out that Definition~\ref{def:qgeneric} is the usual notion of
a generic element for $\G=\PGL_{n/\Q}$ in the sense of Prasad--Rapinchuk.
Let $g\in\G(\Q)$ be regular semisimple, let $G$ be a lift, let
$T=C_\G(g)^\circ$, and let $L$ be the splitting field of $\chi_G$.  Over $L$
we may choose an eigenbasis of $G$, and in this basis $T_L$ is the diagonal
projective torus.  Hence
\[
 X^*(T_L)\cong \Z_0^n:=\{(a_1,\ldots,a_n)\in\Z^n:\textstyle\sum_i a_i=0\}.
\]
If $\sigma\in\Gal(L/\Q)$ satisfies
$\sigma(\alpha_i)=\alpha_{\pi_\sigma(i)}$, then $\sigma$ carries the
$\alpha_i$--eigenline of $G$ onto the $\alpha_{\pi_\sigma(i)}$--eigenline.
Consequently, under the above identification, the Galois action on $X^*(T_L)$
is precisely the permutation action of $\pi_\sigma$ on $\Z_0^n$.  The Weyl
group $W(\G,T)$ is $S_n$ with this same action, and this action is faithful.
Prasad--Rapinchuk call $g$ generic when the Galois image on $X^*(T)$ contains
$W(\G,T)$; see \cite[Definition~9.4 and the discussion following it]{PrasadRapinchukSurvey}.
Since in the present split type-$A$ situation the Galois image is already a
subgroup of $S_n$, this condition is equivalent to
$\pi(\Gal(L/\Q))=S_n$.  By Lemma~\ref{lem:galois}(1), surjectivity also gives
irreducibility of $\chi_G$.  Thus Prasad--Rapinchuk genericity is exactly the
condition in Definition~\ref{def:qgeneric}.  In particular, the distinction
made in \cite[Definition~9.4 and the discussion following it]{PrasadRapinchukSurvey}
between the splitting field of the associated torus and the field generated
by eigenvalues causes no ambiguity here. Since $L$ is generated by the eigenvalues and the permutation action of
$S_n$ on $\Z_0^n$ is faithful, the induced action of $\Gal(L/\Q)$ on
$X^*(T_L)$ is faithful. Thus $L$ is the splitting field of the torus $T$.
\end{remark}

Our Definition~\ref{def:loxodromic} agrees with $\R$--regularity in the sense of \cite[p.~21]{PrasadRapinchukRRegular}.
Indeed, for every $P\in\SL_n(\R)$ the eigenspace of $\Ad(P)$ corresponding to the eigenvalue $1$ is the matrix centralizer of $P$ inside $\mathfrak{sl}_n(\CC)$. The centralizer of any matrix in $M_n(\CC)$ has dimension at least $n$ (this follows immediately from its Jordan normal form), and hence its intersection with $\mathfrak{sl}_n(\CC)$ has dimension at least $n-1$. Thus the number of eigenvalues of $\Ad(P)$ of modulus one, counted with multiplicity, is always at least $n-1$. If $P$ has eigenvalues $\alpha_1,\ldots,\alpha_n$ with pairwise distinct absolute values, then $P$ is diagonalizable over $\CC$. In an eigenbasis, $\Ad(P)$ is the identity on the $(n-1)$--dimensional diagonal Cartan subalgebra, while on the root space $\CC E_{ij}$ it acts by the eigenvalue $\alpha_i/\alpha_j$. Since $|\alpha_i/\alpha_j|\neq1$ for $i\neq j$, the only eigenvalues of modulus one are the $n-1$ copies of $1$ coming from the Cartan subalgebra. Hence the minimum possible number is exactly $n-1$, so $P$ is $\R$--regular in the sense of Prasad--Rapinchuk. Conversely, an $\R$--regular element is semisimple by \cite[p.~21]{PrasadRapinchukRRegular}; after diagonalizing a lift, if $|\alpha_i|=|\alpha_j|$ for some $i\neq j$, then the root eigenvalue $\alpha_i/\alpha_j$ has modulus one in addition to the $n-1$ Cartan eigenvalues, contradicting minimality. Thus the absolute values $|\alpha_1|,\ldots,|\alpha_n|$ are pairwise distinct.

A $\Q$--generic element is regular semisimple. Indeed, let $G\in\SL_n(\Z)$ be a lift of a $\Q$--generic element. By
Definition~\ref{def:qgeneric}(1), the characteristic polynomial $\chi_G$ is
irreducible over $\Q$. Its derivative is nonzero and of smaller
degree, so the greatest common divisor of the polynomial and its derivative is
a constant, and a polynomial with a repeated root shares that root with its
derivative. An irreducible polynomial over a field of
characteristic zero is separable.  Hence $\chi_G$ has $n$ distinct roots in $\CC$. A matrix whose
characteristic polynomial has $n$ distinct roots is diagonalizable, so $G$ is
semisimple with pairwise distinct eigenvalues.
The following rigidity statement is what makes a $\Q$--generic element behave
in $\PGL_n$ as if we were in $\GL_n$.
\begin{lemma}\label{lem:scalar-rigidity}
Let $p\in\Gamma$ be $\Q$--generic and $\R$--regular, and let
$P\in\SL_n(\Z)$ be a lift.
\begin{enumerate}
\item If $U\in\GL_n(\CC)$ and $UPU^{-1}=\zeta P$ for some
$\zeta\in\CC^\times$, then $\zeta=1$.
\item If $u\in\G(\CC)$ commutes with $p$, then every representative
$U\in\GL_n(\CC)$ of $u$ commutes with $P$ as a matrix.
\item One has $C_{M_n(\Q)}(P)=\Q[P]$.
\end{enumerate}
\end{lemma}
\begin{proof}
Let $\alpha_1,\ldots,\alpha_n$ be the eigenvalues of $P$, labelled so that
$|\alpha_1|>\cdots>|\alpha_n|$. If $UPU^{-1}=\zeta P$, then $P$ and
$\zeta P$ have the same spectrum. Comparing the largest absolute values gives
$|\zeta|=1$. The number $\zeta\alpha_1$ is then an eigenvalue of $P$ of
absolute value $|\alpha_1|$. By uniqueness, $\zeta\alpha_1=\alpha_1$, and
hence $\zeta=1$. This proves the first assertion. The second follows from
the definition of projective commutation.
We prove the third assertion directly. Since $\chi_P$ is irreducible of
degree $n$, it is the minimal polynomial of $P$. Fix $0\neq v\in\Q^n$.
If $f(P)v=0$ for a polynomial $f\in\Q[t]$ of degree less than $n$, then
$f=0$. Indeed, if $f\neq0$, the polynomials $f$ and $\chi_P$ are relatively
prime, so there are $a,b\in\Q[t]$ such that $af+b\chi_P=1$. Applying the
resulting matrix identity to $v$ gives $v=0$, a contradiction. Therefore
$v,Pv,\ldots,P^{n-1}v$ is a basis of $\Q^n$.
Let $D\in M_n(\Q)$ commute with $P$. There is a unique polynomial
$f\in\Q[t]$ of degree less than $n$ such that $Dv=f(P)v$. For every
$0\leq j<n$, one has $D(P^jv)=P^jDv=f(P)P^jv$. Thus $D$ and $f(P)$ agree
on a basis, and $D=f(P)$.
\end{proof}
\subsection{Simultaneous generic selection}\label{ssec:selection}
The aim of this subsection is to produce, inside a prescribed finite-index
subgroup, a single element $r$ which makes all the products $r^kc_i$
$\R$--regular simultaneously. Everything needed is already in place. We look
for $r$ in the form $\gamma b\gamma^{-1}$ with $b$ a fixed $\R$--regular
element supplied by Proposition~\ref{prop:bq-lox-dense}; the requirement on
$\gamma$ is that $\gamma$ move the attracting flag of $b$ into general
position with respect to its repelling flag, which by
Lemma~\ref{lem:transversality-open} is the non-vanishing of finitely many
polynomials, so Lemma~\ref{lem:dense-avoids} finds such a $\gamma$ inside
$\Lambda$. Lemma~\ref{lem:flag-dictionary} then translates transversality
into the incidence conditions in every exterior power, and
Lemma~\ref{lem:proximal-product} converts those into proximality of the
products.
\begin{lemma}\label{lem:simultaneous-right}
Let $\Lambda\leq\Gamma$ have finite index, and let
$c_1,\ldots,c_\ell\in\Gamma$. There is an $\R$--regular element
$r\in\Lambda$ such that, for every $i$, the element $r^kc_i$ is
$\R$--regular for all sufficiently large $k$.
\end{lemma}
\begin{proof}
Let $\widetilde\Lambda\leq\SL_n(\Z)$ be the preimage of $\Lambda$. The group $\SL_n(\Z)$ is a lattice in $\SL_n(\R)$, and $\widetilde\Lambda$ has
finite index in it, hence has finite covolume as well and is again a lattice.
The group $\SL_n(\R)$ is connected, semisimple and has no compact factor, so that the Borel density
theorem \cite{borel1960density} (see also \cite[Theorem~3.2.5]{zimmer2013ergodic})
applies and therefore, it is Zariski dense in $\SL_n(\R)$ in the
sense of Definition~\ref{def:zariski-dense}. Fix lifts
$C_1,\ldots,C_\ell\in\SL_n(\Z)$ of $c_1,\ldots,c_\ell$. By
Proposition~\ref{prop:bq-lox-dense} applied to the Zariski-dense subgroup
$\widetilde\Lambda$ of $\SL_n(\R)$, there is a loxodromic
$B\in\widetilde\Lambda$; its image $b\in\Lambda$ is $\R$--regular. Let
$v_1,\ldots,v_n$ be an eigenbasis of $B$ ordered so that the eigenvalues
satisfy $|\mu_1|>\cdots>|\mu_n|$, and let $F^\pm=F^\pm(b)$ be the two flags of
Lemma~\ref{lem:flag-dictionary}. Let $\Delta_{i,j}$ be
the polynomials of Lemma~\ref{lem:transversality-open}, built from this basis
$v_1,\ldots,v_n$ and from $C_1,\ldots,C_\ell$. Using Lemma~\ref{lem:transversality-open}-(3),
there is an element of $\SL_n(\R)$ at which all of them are nonzero. Since
$\widetilde\Lambda$ is Zariski dense in $\SL_n(\R)$,
Lemma~\ref{lem:dense-avoids} provides $Y\in\widetilde\Lambda$ with
 $\Delta_{i,j}(Y)\neq0$ for all $1\leq i\leq\ell$ and  $1\leq j\leq n-1$. Put $R=YBY^{-1}\in\widetilde\Lambda$ and $r=\gamma b\gamma^{-1}\in\Lambda$,
where $\gamma\in\Lambda$ is the image of $Y$. By
Lemma~\ref{lem:conjugate-flags}, the element $r$ is $\R$--regular, its
eigenbasis ordered by decreasing absolute value of the eigenvalues is
$Yv_1,\ldots,Yv_n$, and $F^\pm(r)=\gamma F^\pm$. By part (2) of
Lemma~\ref{lem:transversality-open}, the non-vanishing of $\Delta_{i,j}(Y)$
says that
 $c_i\bigl(F^+(r)\bigr)_j\oplus\bigl(F^-(r)\bigr)_{n-j}=\R^n$ for all $i$ and $j$. That is to say that $c_iF^+(r)\pitchfork F^-(r)$ for every $i$.
Fix
$i$ and $j$. By Lemma~\ref{lem:wedge-attracting} the endomorphism
$\bigwedge^jR$ is proximal, and by Lemma~\ref{lem:flag-dictionary} the
transversality just established gives
\[
 \bigl(\textstyle\bigwedge^jC_i\bigr)
 \bigl(V^+_{\bigwedge^jR}\bigr)\not\subseteq V^<_{\bigwedge^jR}.
\]
These are exactly the hypotheses of Lemma~\ref{lem:proximal-product}, applied
in the vector space $\bigwedge^j\R^n$ with $a=\bigwedge^jR$ and
$c=\bigwedge^jC_i$. Hence
\[
 \textstyle\bigwedge^j(R^kC_i)
 =\bigl(\bigwedge^jR\bigr)^k\bigl(\bigwedge^jC_i\bigr)
\]
is proximal for all sufficiently large $k$. There are only finitely many
pairs $(i,j)$, so a single $k$ works for all of them, and for such $k$,
Lemma~\ref{lem:bq-lox-wedge} says that $R^kC_i$ is loxodromic, that is, that
$r^kc_i$ is $\R$--regular.
\end{proof}
The profinite topology on $\Gamma$ is the group topology in which the
finite-index normal subgroups form a neighborhood basis of the identity.
Thus, if a subset $\Delta\subseteq\Gamma$ is profinite open and
$g\in\Delta$, there is a finite-index normal subgroup
$\Lambda_0\triangleleft\Gamma$ such that $g\Lambda_0\subseteq\Delta$. We need $p$
and $q$ that are simultaneously $\Q$--generic and $\R$--regular and that differ
by a power of the transvection $s$. Genericity is arranged by keeping them both
inside a single coset of a finite-index subgroup, on which it is an
open condition, and $\R$--regularity comes from
Lemma~\ref{lem:simultaneous-right}. We set $s$ to the image of $I_n+E_{12}$.
\begin{proposition}\label{prop:generic-pair}
Let $N\triangleleft\Gamma$ be a finite-index normal subgroup. There are an
integer $m\geq1$ and $\Q$--generic, $\R$--regular elements $p,q\in N$ such
that $q=ps^m$.
\end{proposition}
\begin{proof}
The group $\Gamma$ is finitely generated, so the finite-index subgroup
$N$ is finitely generated. It is Zariski dense by the Borel density
theorem. \cite[Theorem~9.6]{PrasadRapinchukSurvey}
therefore gives a $\Q$--generic element $g\in N$. Let $\Delta\subseteq\Gamma$ be the set of $\Q$--generic elements, so that
$g\in\Delta$. Since every element of
$\Delta$ is regular semisimple, $\Delta$ is exactly the set of regular
semisimple $\Q$--generic elements of $\Gamma$, which is the set denoted
$\Delta(\Gamma,\Q)$ in \cite{PrasadRapinchukOpen}. By
\cite[Theorem~1]{PrasadRapinchukOpen}, $\Delta$ is open in the profinite
topology of $\Gamma$.
Hence, there is a
finite-index subgroup $\Lambda_0\leq\Gamma$ such that
$g\Lambda_0\subseteq\Delta$. Replace $\Lambda_0$ by its
normal core and then intersect it with $N$. We obtain a finite-index normal
subgroup $\Lambda\triangleleft\Gamma$ such that
$\Lambda\subseteq N$ and $g\Lambda\subseteq\Delta$.
The element $s\Lambda$ has finite order in the finite quotient
$\Gamma/\Lambda$. Let $m\geq1$ be this order. Then $s^m\in\Lambda$.
Apply Lemma~\ref{lem:simultaneous-right} to $\Lambda$ and to the two
elements $g$ and $gs^m$. We obtain $r\in\Lambda$ such that both $r^kg$
and $r^kgs^m$ are $\R$--regular for every sufficiently large $k$. Fix
such a $k$, and put $p=r^kg$ and $q=r^kgs^m$.
Since $\Lambda$ is normal, $g^{-1}r^kg\in\Lambda$. Therefore
$p=g(g^{-1}r^kg)\in g\Lambda$. Since $s^m\in\Lambda$, one also has
$q=g(g^{-1}r^kg)s^m\in g\Lambda$. Thus $p,q\in\Delta$, so they are
$\Q$--generic. They are $\R$--regular by the choice of $k$. Finally,
$r,g,s^m\in N$, so $p,q\in N$, and $q=ps^m$.
\end{proof}
\subsection{Centralizers}\label{ssec:centralizers} Let $s\in\Gamma$ be the image of $I_n+E_{12}$. A complex matrix is \emph{unipotent} if
all its eigenvalues are equal to $1$. The matrix $I_n+E_{12}$ is a nontrivial
unipotent matrix since its minimal polynomial is $(t-1)^2$, and hence, it is not
diagonalizable. If $\omega$ is a root of unity, the \emph{order} of $\omega$
is the least positive integer $d$ such that $\omega^d=1$.
\begin{lemma}\label{lem:nontrivial-centralizer-point}
Let $p\in\Gamma$ be $\Q$--generic and $\R$--regular. If
$e\neq z\in C_\Gamma(p)$, then $z$ is regular semisimple. Moreover,
$C_\Gamma(p)\cap C_\Gamma(s)=\{e\}$.
\end{lemma}
\begin{proof}
Choose lifts $P,Z\in\SL_n(\Z)$. By Lemma~\ref{lem:scalar-rigidity}, the
matrices $P$ and $Z$ commute, and the same lemma gives
$C_{M_n(\Q)}(P)=\Q[P]$, so $Z=f(P)$ for some $f\in\Q[t]$.
Let $\alpha_1,\ldots,\alpha_n$ be the eigenvalues of $P$. In an eigenbasis of
$P$, the eigenvalues of $Z$ are $f(\alpha_1),\ldots,f(\alpha_n)$. Suppose
that $f(\alpha_i)=f(\alpha_j)$ for some $i\neq j$. By
Lemma~\ref{lem:galois}(2), for every ordered pair $(a,b)$ with $a\neq b$
there is a Galois automorphism $\sigma$ with $\pi_\sigma(i)=a$ and
$\pi_\sigma(j)=b$. Applying $\sigma$ to $f(\alpha_i)=f(\alpha_j)$, and using
that $f$ has rational coefficients, gives $f(\alpha_a)=f(\alpha_b)$. Hence
all the eigenvalues of $Z$ are equal. The matrix $Z=f(P)$ is diagonal in an
eigenbasis of $P$, so it is diagonalizable, and a diagonalizable matrix with
one eigenvalue is scalar, contrary to $z\neq e$. Therefore the eigenvalues of
$Z$ are pairwise distinct and $z$ is regular semisimple.
It remains to prove that $C_\Gamma(p)\cap C_\Gamma(s)=\{e\}$. Suppose that
$e\neq z$ belongs to this intersection, and let $S=I_n+E_{12}$. Since $z$
commutes projectively with $s$, there is $\omega\in\CC^\times$ with
$SZS^{-1}=\omega Z$. Taking determinants gives
$1=\det(SZS^{-1})=\det(\omega Z)=\omega^n$, because $\det Z=1$; so $\omega$
is a root of unity. Let $d$ be its order. Iterating the relation gives
$S^dZS^{-d}=\omega^dZ=Z$, so $S^d$ commutes with $Z$. The eigenvalues of $Z$
are pairwise distinct, so every matrix commuting with $Z$ is diagonal in an
eigenbasis of $Z$ and is therefore semisimple. On the other hand
$E_{12}^2=0$, so the binomial formula gives $S^d=I_n+dE_{12}$, a nontrivial
unipotent matrix, which is not semisimple. This contradiction proves the
claim.
\end{proof}
\begin{lemma}\label{lem:transvection-powers}
For every nonzero integer $m$, one has $C_\Gamma(s^m)=C_\Gamma(s)$.
\end{lemma}
\begin{proof}
Since $E_{12}^2=0$, $s^m=I_n+mE_{12}$.
Suppose that a lift $X$ commutes projectively with $s^m$. There is a scalar
$\zeta$ such that $X(I_n+mE_{12})X^{-1}=\zeta(I_n+mE_{12})$. The spectrum of the
left-hand side is $\{1\}$, whereas the spectrum of the right-hand side is
$\{\zeta\}$. Thus $\zeta=1$, and then $XE_{12}X^{-1}=E_{12}$. Hence $X$ commutes with
$I_n+E_{12}$, so the corresponding element belongs to $C_\Gamma(s)$. The reverse
inclusion is immediate.
\end{proof}
\begin{corollary}\label{cor:three-centralizers}
Suppose that $p,q\in\Gamma$ are $\Q$--generic and $\R$--regular and that
$q=ps^m$ for some nonzero integer $m$. Then
\[
 C_\Gamma(p)\cap C_\Gamma(s)
 =C_\Gamma(p)\cap C_\Gamma(q)
 =C_\Gamma(s)\cap C_\Gamma(q)
 =\{e\}.
\]
\end{corollary}
\begin{proof}
The first equality is Lemma~\ref{lem:nontrivial-centralizer-point}. If
$z\in C_\Gamma(p)\cap C_\Gamma(q)$, then $z$ centralizes $p^{-1}q=s^m$, so
Lemma~\ref{lem:transvection-powers} gives
$z\in C_\Gamma(p)\cap C_\Gamma(s)$, and hence $z=e$. If
$z\in C_\Gamma(s)\cap C_\Gamma(q)$, then it centralizes both $s^m$ and $q$,
hence also $q(s^m)^{-1}=p$; thus $z\in C_\Gamma(s)\cap C_\Gamma(p)$, and
again $z=e$.
\end{proof}
\subsection{Boundary averaging for generic elements}
This is the last preliminary, and it is the source of every element we shall
extract in Section~\ref{sec:main}. More precisely, for a $\Q$--generic,
$\R$--regular element $p$, the element $E_{C_\Gamma(p)}(a)$ lies in the
norm-closed convex hull of the conjugates
$\lambda(p^k)a\lambda(p^{-k})$. Since a $\Gamma$--invariant subalgebra is
stable under conjugation and norm closed, the expectation therefore maps it
into itself, and that is what lets us move an arbitrary element of $A$ onto a
centralizer without leaving $A$.
Let $X_n=\mathbb P^{n-1}(\R)$. Let $p\in\Gamma$ be $\Q$--generic and
$\R$--regular, and fix a lift $P$. Label its eigenvalues so that
$|\alpha_1|>\cdots>|\alpha_n|$. Choose an eigenvector $v_i$ for
$\alpha_i$, put $x_i=[v_i]\in X_n$, and set
$K(p)=\{x_1,\ldots,x_n\}$.
\begin{lemma}\label{lem:attracting-set}
For every $\mu\in\Prob(X_n)$, the sequence $(p^k)_*\mu$ converges weak-* to
a probability measure supported on $K(p)$.
\end{lemma}
\begin{proof}
For $1\leq i\leq n$, put
$V_i=\operatorname{span}_{\R}\{v_i,\ldots,v_n\}$, with $V_{n+1}=\{0\}$,
and put $Y_i=\mathbb P(V_i)\setminus\mathbb P(V_{i+1})$. These sets form a
Borel partition of $X_n$. If
$x=[a_iv_i+\cdots+a_nv_n]\in Y_i$, where $a_i\neq0$, then
\[
 p^kx=
 \left[v_i+\sum_{j>i}\frac{a_j}{a_i}
 \left(\frac{\alpha_j}{\alpha_i}\right)^kv_j\right]
 \longrightarrow x_i.
\]
For $f\in C(X_n)$, dominated convergence gives
 $\int_{X_n}f\,d((p^k)_*\mu)
 \longrightarrow
 \sum_{i=1}^n\mu(Y_i)f(x_i)$.
Thus $(p^k)_*\mu$ converges to
$\sum_{i=1}^n\mu(Y_i)\delta_{x_i}$.
\end{proof}
\begin{lemma}
\label{lem:centralizer-boundary}
For $t\in\Gamma$, one has $tK(p)\cap K(p)\neq\varnothing$ if and only if
$t\in C_\Gamma(p)$.
\end{lemma}
\begin{proof}
Suppose first that $t\in C_\Gamma(p)$. Choose lifts $T,P\in\SL_n(\Z)$.
By Lemma~\ref{lem:scalar-rigidity}, the matrices $T$ and $P$ commute. Since
$P$ has distinct eigenvalues, $T$ preserves every eigenline of $P$. Hence
$tK(p)=K(p)$.
Conversely, suppose that $tx_i=x_j$ for some $i,j$. Work over the splitting
field $L$ of $\chi_P$. Since a lift of $t$ has rational entries, it commutes
with the coordinatewise action of $\Gal(L/\Q)$. Let
$\mathcal G_i\leq\Gal(L/\Q)\cong S_n$ be the stabilizer of $i$. If
$\sigma\in\mathcal G_i$, then
$\sigma x_j=\sigma(tx_i)=t(\sigma x_i)=tx_i=x_j$. Thus
$\mathcal G_i\leq\mathcal G_j$. Both groups have order $(n-1)!$, so they
are equal. By Lemma~\ref{lem:galois}(2) this forces $i=j$.
Fix $a\in\{1,\ldots,n\}$ and choose $\sigma\in\Gal(L/\Q)$ with
$\pi_\sigma(i)=a$. Then $tx_a=t(\sigma x_i)=\sigma(tx_i)=\sigma x_i=x_a$.
Thus $t$ fixes every eigenline of $P$. A lift of $t$ is diagonal in an
eigenbasis of $P$, so it commutes with $P$. Hence $t\in C_\Gamma(p)$.
\end{proof}
\begin{proposition}
\label{prop:generic-centralizer}
Let $p\in\Gamma$ be $\Q$--generic and $\R$--regular. Then, for every
$a\in C_r^*(\Gamma)$,
 $E_{C_\Gamma(p)}(a)\in
 \overline{\operatorname{conv}}^{\,\|\cdot\|}
 \bigl\{\lambda(p^k)a\lambda(p^{-k}):k\in\mathbb N\bigr\}$.
Consequently, if $A\subseteq C_r^*(\Gamma)$ is a $\Gamma$--invariant
$C^*$--subalgebra, then $E_{C_\Gamma(p)}(A)\subseteq A$.
\end{proposition}
\begin{proof}
Write $K=K(p)$. The element $p^{-1}$ is again $\Q$--generic and
$\R$--regular, and it has the same eigenlines as $p$. Thus
$K(p^{-1})=K$. By the Hahn--Banach separation theorem, it is enough to prove
the following claim. For every bounded linear functional
$\varphi$ on $C_r^*(\Gamma)$, there are a subsequence $(k_j)$ of $\mathbb N$
and a bounded linear functional $\psi$ on $C_r^*(\Gamma)$ such that
\begin{equation}\label{eq:limit-functional}
 \psi(x)=\lim_{j\to\infty}
 \varphi\bigl(\lambda(p^{k_j})x\lambda(p^{-k_j})\bigr)
 \qquad\text{for every }x\in C_r^*(\Gamma),
\end{equation}
and such that $\psi=\varphi\circ E_{C_\Gamma(p)}$.
Consider the reduced crossed product $C(X_n)\rtimes_r\Gamma$ and identify
$C_r^*(\Gamma)$ with the canonical subalgebra generated by the implementing
unitaries. The covariance relation is
$\lambda(t)f\lambda(t)^*=t\cdot f$, where
$(t\cdot f)(x)=f(t^{-1}x)$; see
Brown--Ozawa~\cite[Chapter~4]{BrownOzawa}. By the Hahn--Banach theorem, extend $\varphi$ to a bounded linear functional
$\eta$ on $C(X_n)\rtimes_r\Gamma$. Write
$ \eta=c_1\omega_1-c_2\omega_2+ic_3\omega_3-ic_4\omega_4$,
where $\omega_1,\ldots,\omega_4$ are states and $c_1,\ldots,c_4\ge0$.
For $k\geq1$ and $1\leq i\leq4$, put
$\omega_i^{(k)}=\omega_i\circ\Ad(\lambda(p^k))$. Passing successively to four subsequences, we may choose a single
subsequence $(k_j)$ such that $\omega_i^{(k_j)}$ converges weak-* to a state
$\omega_i'$ for every $1\leq i\leq4$. Let $\nu_i=\omega_i|_{C(X_n)}$. For $f\in C(X_n)$,
\[
 \omega_i^{(k)}(f)=\omega_i(p^k\cdot f)
 =\int_{X_n}f(p^{-k}x)\,d\nu_i(x).
\]
Lemma~\ref{lem:attracting-set}, applied to $p^{-1}$, shows that the
restriction of $\omega_i'$ to $C(X_n)$ is a probability measure supported on
$K$. Let $t\in\Gamma\setminus C_\Gamma(p)$. By
Lemma~\ref{lem:centralizer-boundary}, the finite sets $K$ and $tK$ are
disjoint. Choose $f\in C(X_n)$ such that $0\leq f\leq1$, $f=1$ on $K$, and
$f=0$ on $tK$. Since $\omega_i'(1-f)=0$, the Cauchy--Schwarz inequality for
the state $\omega_i'$ gives
$\omega_i'((1-f)\lambda(t))=0$. Hence
$\omega_i'(\lambda(t))=\omega_i'(f\lambda(t))$. A second application of
Cauchy--Schwarz gives
\[
 \bigl|\omega_i'(f\lambda(t))\bigr|^2
 \leq\omega_i'(\lambda(t)^*f\lambda(t))
 =\omega_i'(t^{-1}\cdot f)=0,
\]
because $t^{-1}\cdot f$ vanishes on $K$. Therefore
$\omega_i'(\lambda(t))=0$. If $t\in C_\Gamma(p)$, then $p^ktp^{-k}=t$ for every $k$. Hence
$\omega_i^{(k)}(\lambda(t))=\omega_i(\lambda(t))$, and therefore
$\omega_i'(\lambda(t))=\omega_i(\lambda(t))$.
Put
$\omega=c_1\omega_1'-c_2\omega_2'+ic_3\omega_3'-ic_4\omega_4'$ and let
$\psi$ be the restriction of $\omega$ to $C_r^*(\Gamma)$. Since $\eta$
restricts to $\varphi$ on $C_r^*(\Gamma)$, equation
\eqref{eq:limit-functional} holds. The preceding two paragraphs show that
$\psi(\lambda(t))=\varphi(\lambda(t))$ for $t\in C_\Gamma(p)$ and
$\psi(\lambda(t))=0$ for $t\notin C_\Gamma(p)$. This is precisely the action
of $\varphi\circ E_{C_\Gamma(p)}$ on the canonical unitaries. Since their
linear span is norm dense in $C_r^*(\Gamma)$, we obtain
$\psi=\varphi\circ E_{C_\Gamma(p)}$, proving the claim.
Now fix $a\in C_r^*(\Gamma)$ and let
 $\mathcal C=
 \overline{\operatorname{conv}}^{\,\|\cdot\|}
 \bigl\{\lambda(p^k)a\lambda(p^{-k}):k\in\mathbb N\bigr\}$.
A standard Hahn-Banach separation argument shows that $E_{C_\Gamma(p)}(a)\in\mathcal C$.
Finally, let $A$ be a $\Gamma$--invariant $C^*$--subalgebra and let
$a\in A$. Every conjugate $\lambda(p^k)a\lambda(p^{-k})$ belongs to $A$.
Since $A$ is convex and norm closed, it contains their norm-closed convex
hull. Therefore $E_{C_\Gamma(p)}(a)\in A$.
\end{proof}
\section{Proof of the main result}\label{sec:main}
We begin by proving an escape result. Apart from the terms which commute with $s$, nothing stays
inside $C_\Gamma(q)$ for infinitely many $k$. The proof is elementary because
$s$ is unipotent, so that $s^{\pm k}$ depends polynomially on $k$ and the whole
question becomes one about polynomials in one variable.
\begin{lemma}\label{lem:polynomial-escape}
Let $q\in\Gamma$ be $\Q$--generic and $\R$--regular, and fix $c\in\Gamma$ and
$t\in\Gamma\setminus C_\Gamma(s)$. Then $cs^{-k}ts^k\in C_\Gamma(q)$ for only
finitely many integers $k\geq0$.
\end{lemma}
\begin{proof}
Put $N=E_{12}$ and $S=I_n+N$. Since $N^2=0$, we have $S^k=I_n+kN$ and $S^{-k}=I_n-kN$ for every integer $k$.  Choose lifts $C,T,Q\in\operatorname{GL}_n(\mathbb C)$
of $c,t,q$, respectively.  Since $q$ is $\mathbb Q$--generic and $\mathbb R$--regular, a lift $Q$ of $q$
has pairwise distinct eigenvalues over $\mathbb C$, and its centralizer in
$\operatorname{PGL}_n(\mathbb C)$ is the diagonal projective torus in an
eigenbasis of $Q$.
Thus there exists
$U\in\operatorname{GL}_n(\mathbb C)$ such that $U^{-1}QU$ is diagonal and such
that an element of $\operatorname{PGL}_n(\mathbb C)$ belongs to
$C_\G(q)(\mathbb C)$ precisely when, in this basis, it is represented by a
diagonal matrix.
Suppose, towards a contradiction, that
$cs^{-k}ts^k\in C_\Gamma(q)$
for infinitely many integers $k\geq 0$.  Define
 $B(z)=U^{-1}C(I_n-zN)T(I_n+zN)U
 ~(z\in\mathbb{C}).$
Every entry of $B(z)$ is a polynomial in $z$.  For every integer $k\geq0$,
the projective class of $B(k)$ is $cs^{-k}ts^k$.  Hence, for infinitely many
such $k$, the matrix $B(k)$ is diagonal.  If $i\ne j$, then the $(i,j)$-entry
of $B(z)$ is therefore a polynomial in one variable with infinitely many
roots, hence is identically zero.  Thus $B(z)$ is diagonal for every $z\in\mathbb C$.
Write $B(z)=\operatorname{diag}(b_1(z),\ldots,b_n(z)).$
Since $N^2=0$, both $I_n-zN$ and $I_n+zN$ have determinant one.  Consequently,
 $b_1(z)\cdots b_n(z)=\det B(z)=\det(C)\det(T),$
which is a nonzero constant.  Since the $b_i(z)$ are polynomials, they are all
constant.  Hence $B(z)$, and therefore
 $C(I_n-zN)T(I_n+zN),$
are constant in $z$.  Comparing the values at $z=0$ and $z=1$ yields
$C(I_n-N)T(I_n+N)=CT.$
After cancelling $C$ and using $(I_n+N)^{-1}=I_n-N$, we obtain $S^{-1}TS=T.$
Thus $t\in C_{\Gamma}(s)$, contradicting the assumption on $t$.  Therefore,
$cs^{-k}ts^k\in C_\Gamma(q)$ for only finitely many integers $k\geq0$.
\end{proof}
\begin{remark}
    In the application $c$ will lie in $C_\Gamma(p)$, but the proof uses no
    assumption on $c$ whatsoever.
\end{remark}
We are now ready to prove the main theorem. We begin by extracting a group unitary. The idea is to approximate $\lambda(p)$ from inside $A$, to push the approximants onto the centralizers
of $p$ using Proposition~\ref{prop:generic-centralizer}, and then to
read off $\lambda(q)$ as the surviving term of the partial-matching formula. Once that is done, a Galois correspondence result completes the proof. This is Theorem~\ref{thm:main} from the introduction. Note that if $A=\mathbb{C}$, then it is $C_r^*(\{e\})$ and hence, comes from a normal subgroup.
\begin{theorem}
Let $A\subseteq C_r^*(\Gamma)$ be nontrivial, unital and $\Gamma$--invariant. Then, $A$ is of the form $C_r^*(K)$ for some normal subgroup $K$ of $\Gamma$.
\end{theorem}
\begin{proof}
It follows from \cite{KalantarPanagopoulos} that $A''=L(N_0)$ for a nontrivial normal subgroup $N_0\triangleleft\Gamma$.\\
\noindent
\textit{Step-1:}
$A$ contains a nontrivial group element.\\
\noindent
Note that the subgroup $N_0$ has finite index, since $N_0$ is nontrivial and
Margulis's normal subgroup theorem applies~\cite[Chapter~IV]{MargulisBook}.
Choose $p,q\in N_0$ and $m\geq1$ given by
Proposition~\ref{prop:generic-pair}. Thus $q=ps^m$, and $p,q$ are
$\Q$--generic and $\R$--regular. By Corollary~\ref{cor:three-centralizers}, the three groups $C_\Gamma(p)$,
$C_\Gamma(q)$ and $C_\Gamma(s)$ have pairwise trivial intersections. It follows from Lemma~\ref{lem:polynomial-escape} that the third condition $(C)$ of Proposition~\ref{prop:strategy-main} is met and hence, all the conditions are satisfied by $p$, $q$ and $s$. Let $0<\varepsilon<3/4$.
Since $p,s^m\in N_0$, the unitaries $\lambda(p)$ and $\lambda(s^m)$ belong
to $A''=L(N_0)$. By Kaplansky's density theorem, there are sequences of contractions
$(a_j)_{j\geq1}$ and $(d_j)_{j\geq1}$ in $A$ such that
$\|a_j-\lambda(p)\|_2\to0$ and $\|d_j-\lambda(s^m)\|_2\to0$. Fix $j$ such
that
\[
 \|a_j-\lambda(p)\|_2<\frac{\varepsilon}{3}
 \qquad\text{and}\qquad
 \|d_j-\lambda(s^m)\|_2<\frac{\varepsilon}{3}.
\]
By Proposition~\ref{prop:generic-centralizer},
$E_{C_\Gamma(p)}(a_j)$ belongs to $A$. Consequently,
\begin{equation}\label{eq:b-close}
 \|E_{C_\Gamma(p)}(a_j)-\lambda(p)\|_2
 \leq\|a_j-\lambda(p)\|_2<\frac{\varepsilon}{3}.
\end{equation}
Similarly,
\begin{equation}\label{eq:x-close}
 \|E_{C_\Gamma(s)}(d_j)-\lambda(s^m)\|_2
 \leq\|d_j-\lambda(s^m)\|_2<\frac{\varepsilon}{3}.
\end{equation}
For every $k$, the element
$E_{C_\Gamma(p)}(a_j)\lambda(s^{-k})d_j\lambda(s^k)$ belongs to $A$. Since
$q$ is $\Q$--generic and $\R$--regular,
Proposition~\ref{prop:generic-centralizer} gives
$E_{C_\Gamma(q)}(E_{C_\Gamma(p)}(a_j)\lambda(s^{-k})d_j\lambda(s^k))\in A$.
By Proposition~\ref{prop:strategy-main} and the norm closedness of
$A$, we see that
\[
 \lim_{k\to\infty}
 E_{C_\Gamma(q)}\left(
 E_{C_\Gamma(p)}(a_j)\lambda(s^{-k})d_j\lambda(s^k)\right)
 =E_{C_\Gamma(q)}\left(
 E_{C_\Gamma(p)}(a_j)E_{C_\Gamma(s)}(d_j)\right)\in A.
\]
By the Cauchy--Schwarz inequality and \eqref{eq:b-close}, we see that
\begin{equation}
\label{eq:coeff-b}
\begin{aligned}
 \left|\tau_0(E_{C_\Gamma(p)}(a_j)\lambda(p)^*)-1\right|
 &=\left|\tau_0((E_{C_\Gamma(p)}(a_j)-\lambda(p))\lambda(p)^*)\right|
 \\&\leq\|E_{C_\Gamma(p)}(a_j)-\lambda(p)\|_2
 <\frac{\varepsilon}{3}.
 \end{aligned}
\end{equation}
Similarly, we see that
\begin{equation}\label{eq:coeff-x}
 \left|\tau_0(E_{C_\Gamma(s)}(d_j)\lambda(s^m)^*)-1\right|
 \leq\|E_{C_\Gamma(s)}(d_j)-\lambda(s^m)\|_2
 <\frac{\varepsilon}{3}.
\end{equation}
Moreover, using equations~\eqref{eq:b-close} and~ \eqref{eq:coeff-b}, we see that
\begin{equation}\label{eq:tail-b}
\begin{aligned}
&\left\|E_{C_\Gamma(p)}(a_j)-
 \tau_0(E_{C_\Gamma(p)}(a_j)\lambda(p)^*)\lambda(p)\right\|_2\\
&\quad\leq
 \|E_{C_\Gamma(p)}(a_j)-\lambda(p)\|_2
 +\left|1-\tau_0(E_{C_\Gamma(p)}(a_j)\lambda(p)^*)\right|
 <
 \frac{2\varepsilon}{3},
\end{aligned}
\end{equation}
and, similarly, by virtue of equations~\eqref{eq:x-close} and \eqref{eq:coeff-x}, we also have that
\begin{equation}\label{eq:tail-x}
 \left\|E_{C_\Gamma(s)}(d_j)-
 \tau_0(E_{C_\Gamma(s)}(d_j)\lambda(s^m)^*)\lambda(s^m)\right\|_2
 <
 \frac{2\varepsilon}{3}.
\end{equation}
Since $d_j$ is a contraction, one has
$|\tau_0(E_{C_\Gamma(s)}(d_j)\lambda(s^m)^*)|\leq1$, and therefore
\begin{align*}
 \left|\tau_0(E_{C_\Gamma(p)}(a_j)\lambda(p)^*)
 \tau_0(E_{C_\Gamma(s)}(d_j)\lambda(s^m)^*)-1\right|
 &\leq
 \left|\tau_0(E_{C_\Gamma(p)}(a_j)\lambda(p)^*)-1\right|
 \\&+\left|\tau_0(E_{C_\Gamma(s)}(d_j)\lambda(s^m)^*)-1\right|.
\end{align*}
Applying Lemma~\ref{lem:fourier-matching} with
$b=E_{C_\Gamma(p)}(a_j)$ and
$x=E_{C_\Gamma(s)}(d_j)$, along with equations~\eqref{eq:coeff-b},\,\eqref{eq:coeff-x},\,
 \eqref{eq:tail-b}, and \eqref{eq:tail-x}, we see that
 \[
\begin{gathered}
 \Bigl\|E_{C_\Gamma(q)}\bigl(E_{C_\Gamma(p)}(a_j)E_{C_\Gamma(s)}(d_j)\bigr)-\lambda(q)\Bigr\|\\[2pt]
 \begin{aligned}
 &\leq\Bigl\|E_{C_\Gamma(p)}(a_j)-\tau_0\bigl(E_{C_\Gamma(p)}(a_j)\lambda(p)^*\bigr)\lambda(p)\Bigr\|_2\\
 &\qquad\times
 \Bigl\|E_{C_\Gamma(s)}(d_j)-\tau_0\bigl(E_{C_\Gamma(s)}(d_j)\lambda(s^m)^*\bigr)\lambda(s^m)\Bigr\|_2\\
 &\quad+\bigl|\tau_0\bigl(E_{C_\Gamma(p)}(a_j)\lambda(p)^*\bigr)-1\bigr|
 +\bigl|\tau_0\bigl(E_{C_\Gamma(s)}(d_j)\lambda(s^m)^*\bigr)-1\bigr|\\
 &<\frac{4\varepsilon^2}{9}+\frac{\varepsilon}{3}+\frac{\varepsilon}{3}<\varepsilon .
 \end{aligned}
\end{gathered}
\]
where the last inequality uses $\varepsilon<3/4$. Every element on the
left belongs to $A$, and $\varepsilon$ was arbitrary, so
$\lambda(q)\in A$. The element $q$ is nontrivial because it is
$\R$--regular. This finishes the proof of the first step.
It follows from \textit{Step-1} that there exists $q\neq e$ such that
$\lambda(q)\in A$.  Let $N=\langle\!\langle q\rangle\!\rangle_\Gamma$ be the normal closure
of $q$ in $\Gamma$. Therefore
$C_r^*(N)\subseteq A\subseteq C_r^*(\Gamma)$. Again, by Margulis's normal subgroup theorem, $N$ has finite index in $\Gamma$.  
Let $F=\Gamma/N$, and let $\pi:\Gamma\to F$ be the quotient map. Choose
a section $\rho:F\to\Gamma$ such that $\rho(e)=e$. For $x\in F$ and
$l\in N$, let
$\beta_x(\lambda(l))
 =\lambda(\rho(x)l\rho(x)^{-1})$ and
$v(x,y)=\lambda(\rho(x)\rho(y)\rho(xy)^{-1})$.
Using~\cite[Theorem~2.1]{Bedos1991}, we see that $(\beta,v)$ is a
cocycle action of $F$ on $C_r^*(N)$ and gives an isomorphism
 $\Phi:C_r^*(N)\rtimes_{\beta,r}^{\,v}F\longrightarrow C_r^*(\Gamma)$.
This isomorphism is identity on $C_r^*(N)$ and sends the canonical
unitary associated with $x\in F$ to $\lambda(\rho(x))$.
We verify the two hypotheses needed to apply the twisted Galois
correspondence. Since $C_r^*(\Gamma)$ is simple (see~\cite{BekkaCowlingHarpe}), using \cite[Theorem~1.4]{breuillard2017c}, we see that $C_r^*(N)$ is simple.  Second, $\beta_x$ is outer for every $x\neq e$. Since $N$ has finite index in $\Gamma$, it is plump in $\Gamma$ in the sense of \cite{amrutam2021intermediate} and hence, $C_{\Gamma}(N)=\{e\}$.  Fix $x\neq e$ and put $g=\rho(x)$. Conjugation by $g$ defines an
automorphism $\theta_g$ of $N$. This automorphism is outer. Indeed, if
there were $u\in N$ such that $glg^{-1}=ulu^{-1}$ for every $l\in N$,
then $u^{-1}g\in C_\Gamma(N)=\{e\}$, and hence $g=u\in N$. This would
contradict $x\neq e$.
Since $C_r^*(N)$ is simple, the group $N$ is i.c.c. It follows from
\cite[Lemma~3.4]{Bedos1991} that for every $u\in N$,
the set
 $\{\theta_g(h)uh^{-1}:h\in N\}$
is infinite. Applying~\cite[Lemma~3.3]{Bedos1991}, we get that the induced
automorphism $\beta_x$ of $C_r^*(N)$ is freely acting. This means that the
only element $b\in C_r^*(N)$ satisfying $\beta_x(a)b=ba$ for every
$a\in C_r^*(N)$ is $b=0$. In particular, $\beta_x$ is outer. Let $\widetilde A=\Phi^{-1}(A)$.
\cite[Theorem~4.4]{CameronSmithGalois} gives a subgroup $L\leq F$ such
that
$ \widetilde A=C_r^*(N)\rtimes_{\beta,r}^{\,v}L$.
Letting $K=\pi^{-1}(L)$, it follows that
$A=C_r^*(K)$.
\end{proof}
\section*{Acknowledgements}
This paper grew out of joint work with Yongle Jiang, Yair Glasner, and Yair Hartman on boomerang
subalgebras, which will appear separately. I thank all three of them for many discussions, and Jiang in particular, who suggested running the argument through the transvection $s=I_n+E_{12}$. Sometimes one needs a friend's perspective on purpose and why we must do what we do. The author thanks Jagoda for that. 

An earlier version of this paper used a different extraction argument. Its
general strategy revolved around three conditions. First, $\mathrm{(F1)}$, requiring
the relevant centralizers to have trivial intersection; Second, $\mathrm{(F2)}$,
requiring a conjugate product to remain $\Q$--generic and $\R$--regular; and
$\mathrm{(F3)}$, requiring a certain intersection of a product of centralizers
with the centralizer of the product to be finite. The purpose of
$\mathrm{(F3)}$ was to produce an element of $A$ with finite Fourier support,
and a second turn, formulated as $\mathrm{(F4)}$, then isolated a single group
unitary. That route is recorded separately in \cite{AmrTwoTurn} and will hopefully be used elsewhere.

That earlier draft was discussed with OpenAI's ChatGPT, using GPT--5.6
Thinking, and the model proposed bypassing the finite-support argument and
$\mathrm{(F4)}$ altogether. With Jiang's transvection $s$ held fixed,
it suggested to construct $\Q$--generic $\R$--regular elements $p,q$ with $q=ps^m$, and
replace $\mathrm{(F3)}$ by an escape property together with a matched
Fourier-series argument. This suggestion led to
Proposition~\ref{prop:strategy-main} and Lemma~\ref{lem:fourier-matching},
which form the analytic core of the present proof. I used Anthropic's Claude
during the writing of the paper, for discussion of exposition and assistance
with language and \LaTeX. I developed the construction, supplied, and verified
all proofs, and checked the mathematical arguments and references, and I take
full responsibility for the contents of the paper.
\bibliographystyle{alpha}
\bibliography{ISR}
\end{document}